\documentclass[12pt]{article}
\usepackage{graphicx} 

\usepackage{mathtools}
\usepackage{amsmath}
\usepackage{amssymb}
\usepackage{amsthm}
\usepackage{tikz}
\usepackage{tikz,pgfplots}
\usetikzlibrary{decorations.markings}
\usepackage{comment}
\usepackage{url}
\usepackage{physics}
\usepackage{subfig}

\usepackage{graphicx,tikz}
\usetikzlibrary{shadows}
\usepackage{caption}
\usepackage{todonotes}
\usepackage{graphicx}
\usetikzlibrary{arrows}
\usepackage[utf8]{inputenc}
\usepackage[english]{babel}
\usepackage{array}
\usepackage{booktabs}
\usepackage{mdframed}
\usepackage{enumitem}

\usepackage{soul}

\usepackage{xcolor}
\usepackage{hyperref}

\hypersetup{
  colorlinks   = true, 
  urlcolor     = blue, 
  linkcolor    = blue, 
  citecolor   = red 
}
\setstcolor{red}
\newtheorem{theorem}{Theorem}

\newtheorem{corollary}[theorem]{Corollary}
\newtheorem{lemma}[theorem]{Lemma}
\newtheorem{conjecture}[theorem]{Conjecture}

\newtheorem{proposition}[theorem]{Proposition}

\newcommand{\marrow}{\marginpar[\hfill$\longrightarrow$]{$\longleftarrow$}}

\newcommand{\niceremarkcolor}[4]{\textcolor{#4}{\textsc{#1 #2:} \marrow\textsf{#3}}}
\newcommand{\aida}[2][says]{\niceremarkcolor{Aida}{#1}{#2}{blue}}
\newcommand{\luuk}[2][says]{\niceremarkcolor{Luuk}{#1}{#2}{orange}}

\title{Eigenvalue bounds for the distance-$t$ chromatic number of a graph and their application to Lee codes}
\author{Aida Abiad\thanks{\texttt{a.abiad.monge@tue.nl}, Department of Mathematics and Computer Science, Eindhoven University of Technology, The Netherlands} \and
Alessandro Neri\thanks{\texttt{alessandro.neri@ugent.be}, Department of Mathematics: Analysis, Logic and Discrete Mathematics, Ghent University, Belgium}
\and
Luuk Reijnders\thanks{\texttt{l.e.r.m.reijnders@student.tue.nl}, Department of Mathematics and Computer Science, Eindhoven University of Technology, The Netherlands}}
\date{}

\begin{document}

\maketitle

\begin{abstract}
We derive eigenvalue bounds for the $t$-distance chromatic number of a graph, which is a generalization of the classical chromatic number.  We apply such bounds to hypercube graphs, providing alternative spectral proofs for results by Ngo, Du and Graham [\emph{Inf. Process. Lett.}, 2002], and improving their bound for several instances. We also apply the eigenvalue bounds to Lee graphs, extending results by Kim and Kim [\emph{Discrete Appl. Math.}, 2011]. Finally, we provide a complete characterization for the existence of perfect Lee codes of minimum distance $3$. In order to prove our results, we use a mix of spectral and number theory tools. Our results, which provide the first application of spectral methods to Lee codes, illustrate that such methods succeed to capture the nature of the Lee metric.
\end{abstract}




\section{Introduction}

In the past, various distance based colorings have been investigated in the literature. Distance based coloring was first studied in 1969 by Kramer and Kramer \cite{KK1969,KK1969v2}, when they introduced the notion of a \emph{$t$-distance coloring}, for some natural number $t$. In this type of coloring, we require that vertices at distance at most $t$ receive distinct colors. When $t=1$, we recover the classical vertex coloring.  Throughout the years, $t$-distance colorings, in particular the case when $t=2$, became a focus for many researchers; see e.g. \cite{BI2012, H2009,HHDR2017,T2018}.

The \emph{distance-$t$ chromatic number} of a graph $G$, denoted as $\chi_t(G)$, is the smallest number of colours required in a $t$-distance coloring.
For a positive integer $t$, the \emph{$t$-th power of a graph} $G=(V,E)$ on, denoted by $G^t$, is a graph with vertex set $V$ in which two distinct elements of $V$ are joined by an edge if there is a path in $G$ of length at most $t$ between them. Thus, the distance-$t$ chromatic number is equivalently defined as the chromatic number of the $t$-th power graph, that is, $\chi_t(G) = \chi_1(G^t)$. However, even the simplest algebraic or combinatorial parameters of power graph $G^t$ cannot be deduced easily from the corresponding parameters of the graph $G$. 
For instance, neither the spectrum~\cite{Das2013LaplacianGraph},~\cite[Section 2]{ACFNS2022}, nor the average degree~\cite{Devos2013AveragePowers}, nor the rainbow connection number~\cite{Basavaraju2014RainbowProducts} of $G^t$ can be derived in general directly from those of the original graph $G$.

The above, plus the fact that this parameter is known to be NP-hard to compute \cite{LS1995}, provides the initial motivation for the first part of this work, where several eigenvalue polynomial bounds on $\chi_t(G)$ will be derived and optimized, analogously as it was done for the $t$-\emph{independence number} $\alpha_t(G)$ (largest maximum size of a set of vertices at pairwise distance greater than $t$) in \cite{ACF2019}. In particular, we extend the best Ratio-type bound on $\alpha_2(G)$ for regular graphs \cite[Corollary 3.3]{ACF2019} to $\chi_2(G)$ of general graphs, and show its optimality.

In the second part of this paper, we present several applications of the new spectral bounds on $\chi_t(G)$ to coding theory. In particular, we present several eigenvalue bounds for $\chi_t(Q_n)$ of the hypercube graph $Q_n$, which turn out to give alternative proofs and also bound improvements to known results by Ngo, Du and Graham \cite{ngo_new_2002}. Next, we focus on a generalization of hypercube graphs, the so-called Lee graphs, which are associated with the Lee metric. We investigate combinatorial and spectral properties of Lee graphs, and use them to provide the first eigenvalue bounds for $\chi_2(G(n,q))$ of the Lee graph $G(n,q)$, extending results by Kim and Kim \cite{kim_2-distance_2011}.
We computationally show  that our bounds are tight for several instances, and that they are competitive when we compare them with the optimization bounds for Lee codes previously obtained via linear programming (LP) by Astola and Tabus \cite{astola_bounds_2013} and via semidefinite programming (SDP) by Polak \cite{polak_semidefinite_2019}. Finally, we provide a complete characterization for the existence of (not necessarily linear) perfect Lee codes of minimum distance $3$. Our results demonstrate that spectral methods succeed to capture information about the Lee metric.

This paper is structured as follows. In Section \ref{sec:spec_chit} we provide an overview of the existing spectral bounds on $\chi_t(G)$, and we prove an extension of one of them to non-regular graphs. In Section \ref{sec:chit_Qn} we use the spectral bounds from Section \ref{sec:spec_chit} to obtain an alternative bound on $\chi_t(Q_n)$. While for $t\in\{2,3\}$  we obtain an alternative proof to the bound on $\chi_t(Q_n)$ by Ngo, Du and Graham \cite{ngo_new_2002}, for $t\in\{4,5\}$ our spectral bound sometimes outperforms it. Next, in Section \ref{sec:Gnq_properties}, we show several combinatorial and spectral properties of the  Lee graphs, and using number theory tools, we derive some novel characterizations of the adjacency spectrum of these graphs. In Section \ref{sec:chit_Gnq} we apply our spectral bound on $\chi_t(G)$ from Section \ref{sec:spec_chit} to Lee graphs, extending results by Kim and Kim \cite{kim_2-distance_2011}. Finally, we provide a complete characterization for the existence of (not necessarily linear) perfect Lee codes of minimum distance $3$ in Section \ref{sec:existenceperfectLeecodes}.

\section{Eigenvalue bounds for the distance chromatic number of a graph}\label{sec:spec_chit}

In this section we give an overview of several spectral (Ratio-type) bounds on $\chi_t(G)$. This was first investigated in \cite{ACF2019} for regular graphs and for the so-called $t$-\emph{independence number} $\alpha_t(G)$ (independence number of $G^t$), and this result was extended in \cite{ACFNS2022} for $\chi_t(G)$ for general graphs. Indeed, the distance-$t$ chromatic number of a graph is directly related to the $t$-independence number via the following inequality: 
\begin{equation}\label{lem:alpha_chi_relation}
    \chi_t(G) \geq \frac{|V|}{\alpha_t(G)}. 
\end{equation}

In addition to this general bound on $\chi_t(G)$, we obtain a slightly stronger version which holds for regular graphs; see Theorem \ref{thm:ratio_chi_reg}. Moreover, we also present closed formulas for the best (general) Ratio-type bounds for $\chi_2(G)$ (see Theorem \ref{thm:ratio_chi_2} and Corollary \ref{cor:ratio_chi_2_reg}) and $\chi_3(G)$ (see Corollary \ref{cor:ratio_chi_3_reg}), based on similar results for $\alpha_2(G)$ and $\alpha_3(G)$ from \cite{ACF2019} and \cite{kavi_optimal_2023}, respectively.

    We start by stating a Ratio-type bound for $\chi_t(G)$ which holds for general graphs and which was shown by Abiad et al. \cite{ACFNS2022} using weight partitions and interlacing. Denote by $\mathbb{R}_t[x]$ the set of polynomials in the variable $x$ with coefficients in $\mathbb{R}$ and degree at most $t$.
    
    \begin{theorem}[General Ratio-type bound {\cite[Theorem 4.3]{ACFNS2022}}]\label{thm:ratio_chi}
        Let $G$ be a graph with $n$ vertices and eigenvalues $\lambda_1 \geq \lambda_2 \geq \cdots \geq \lambda_n$ and adjacency matrix $A$. Let $p \in \mathbb{R}_t[x]$ with corresponding parameters $W(p) := \max_{u \in V}$ $\{(p(A))_{uu}\}$ and $ \lambda(p):= \min_{i\in[2,n]} \{p(\lambda_i)\}$, and assume $p(\lambda_1) > \lambda(p)$. Then, 
        \begin{equation}\label{eq:ratio_chi}
            \chi_t(G) \geq \frac{p(\lambda_1) - \lambda(p)}{W(p) - \lambda(p)}.
        \end{equation}
    \end{theorem}

    In addition to Theorem \ref{thm:ratio_chi} which holds for general graphs, we can also obtain a slightly stronger bound for regular graphs, by using the original Ratio-type bound on $\alpha_t(G)$ \cite[Theorem 3.2]{ACF2019}, followed by the relation \eqref{lem:alpha_chi_relation}. This will be the bound that we will use in coming sections for the applications to coding theory.
    
    \begin{theorem}[Ratio-type bound]\label{thm:ratio_chi_reg}
        Let $G$ be a regular graph with $n$ vertices and eigenvalues $\lambda_1 \geq \lambda_2 \geq \cdots \geq \lambda_n$ and adjacency matrix $A$. Let $p \in \mathbb{R}_t[x]$ with corresponding parameters $W(p) := \max_{u \in V}$ $\{(p(A))_{uu}\}$ and $ \lambda(p):= \min_{i\in[2,n]} \{p(\lambda_i)\}$, and assume $p(\lambda_1) > \lambda(p)$. Then, the distance-$t$ chromatic number of $G$ satisfies the bound
        \begin{equation}\label{eq:ratio_chi_regular}
            \chi_t(G) \geq \frac{n}{\left\lfloor n\frac{W(p) - \lambda(p)}{p(\lambda_1) - \lambda(p)}\right\rfloor}. 
        \end{equation}
    \end{theorem}
    \begin{proof}
        Apply \cite[Theorem 3.2]{ACF2019} to $G$, followed by \eqref{lem:alpha_chi_relation}.
    \end{proof}

\subsection{Optimization of the Ratio-type bounds}  
In the following, we propose the linear optimization implementation of the previously seen Ratio-type bounds on $\chi_t(G)$ from Theorem \ref{thm:ratio_chi} and Theorem \ref{thm:ratio_chi_reg}. The latter will be one of the main tools in Sections \ref{sec:chit_Qn} and \ref{sec:chit_Gnq}. 
\subsubsection*{LP for the General Ratio-type bound for $\chi_t(G)$ (Theorem \ref{thm:ratio_chi})}\label{LP:chitgeneralgraph}

    Let $G = (V,E)$ have adjacency matrix $A$ and distinct eigenvalues $\theta_0 > \cdots > \theta_d$. Note that we can scale by a positive number and translate the polynomial used in Theorem \ref{thm:ratio_chi} without changing the value of the bound. Hence, we can assume $W(p) -\lambda(p) =1$. Furthermore, ~$\lambda(p_k) < W(p_k)$, so the scaling does not flip the sign of the bound. Hence, the problem reduces to finding the $p$ which maximizes $p(\lambda_1) - \lambda(p)$. For each~$u\in V$ and~$\ell \in [1, d]$, assume that~$W(p_k) = (p_k(A))_{uu}$,~$0=\lambda(p_k) = p_k(\theta_\ell)$ and solve the Linear Program (LP) below. The maximum of these~$dn$ solutions then equals the best possible bound obtained by Theorem \ref{thm:ratio_chi}.

   {\small{ 
    \begin{equation}\label{eq:ratio_chi_LP}
        \boxed{\begin{aligned}
            \text{variables: } &(a_0, \hdots,a_t) \\
            \text{input: }&\text{The adjacency matrix } A \text{ and eigenvalues } \{\theta_0,\hdots,\theta_d\} \text{ of a graph }G. \\
            &\text{A vertex } u \in V, \text{ an } \ell \in [1,d]. \text{ An integer } t.\\
            \text{output: }&(a_0, \hdots,a_t) \text{, the coefficients of a polynomial } p \\ 
            \ \\
            \text{maximize} \ \ &\sum_{i=0}^t a_i \theta_0^i -\sum_{i=0}^t a_i\theta_\ell^i\\
            \text{subject to} \ \ & \sum_{i=0}^t a_i((A^i)_{vv} - (A^i)_{uu}) \leq 0 , \ v \in V \setminus \{u\} \\
            &\sum_{i=0}^t a_i((A^i)_{uu} - \theta_\ell^i) = 1\\
            &\sum_{i=0}^t a_i(\theta_0^i - \theta_j^i) > 0, \ \ j \in [1,d] \\
            &\sum_{i=0}^t a_i(\theta_j^i - \theta_\ell^i) \geq 0, \ \ j \in [1,d]
        \end{aligned}}
    \end{equation}
    }}
    
    Here the objective function is simply $p(\lambda_1) - \lambda(p)$. The first constraint says $(p(A))_{uu} \geq (p(A))_{vv}$ for all vertices $v \neq u$, which ensures $W(p) = (p(A))_{uu}$. The second constraint gives $p$ the correct scaling and translation such that $W(p) - \lambda(p) = 1$. The third constraint says $p(\theta_0) > p(\theta_j)$ for all $j \in [1,d]$, which ensures $p(\lambda_1) > \lambda(p)$. And the final constraint says $p(\theta_\ell) \leq p(\theta_j)$ for all $j \in [1,d]$, which ensures $\lambda(p) = p(\theta_\ell)$.

\subsubsection*{LP for the Ratio-type bound for $\chi_t(G)$ (Theorem \ref{thm:ratio_chi_reg}) for walk-regular graphs} \label{minorpolysLP:chitwrgraph} 
 
We have already seen that the assumption of regularity let us obtain Theorem \ref{thm:ratio_chi_reg}, which is slightly stronger of Theorem \ref{thm:ratio_chi}, the latter of which holds in general. In this same vein, regularity assumptions allow us to obtain a Linear Program (LP) \eqref{eq:ratio_chi_LP_walk} for Theorem \ref{thm:ratio_chi_reg}, which in practice is much faster than LP \eqref{eq:ratio_chi_LP}, computationally speaking.

The LP we are going to describe is analogous to the LP from \cite{F2020} for $\alpha_t(G)$, but is repeated here for completeness. As such, it also only applies to $t$-partially walk-regular graphs. A graph $G$ is called {\em $t$-partially walk-regular} for some integer $t\ge 0$, if the number of closed walks of a given length $l\le t$, rooted at a vertex $v$, only depends on $l$. In other words, if $G$ is a $t$-partially walk-regular graph, then for any polynomial $p\in \mathbb{R}_t[x]$ the diagonal of $p(A)$ is constant with entries.
Given a $t$-partially walk regular graph with adjacency matrix $A$, for every $p\in \mathbb R_t[x]$ the diagonal of $p(A)$ is constant with entries
$$
(p(A))_{uu}=\frac{1}{n}\tr p(A)=\frac{1}{n}\sum_{i=1}^n  p(\lambda_i)\quad \mbox{for all }u\in V.
$$

For instance, every (simple) graph is $t$-partially walk-regular for $t\in\{0,1\}$ and every regular graph is $2$-partially walk-regular.

Through the following LP, we are going to compute so-called \emph{$t$-minor polynomials}, defined by Fiol \cite{F2020}. For a graph $G$ with adjacency matrix $A$ and distinct eigenvalues $\theta_0 > \cdots > \theta_d$, consider the set of real polynomials $\mathcal{P}_t = \{p \in \mathbb{R}_t[x] : p(\theta_0) = 1, p(\theta_i) \geq 0,$ for $i \in [1,d]\}$. Then, a \emph{$t$-minor polynomial} of $G$ is a polynomial $p_t \in \mathcal{P}_t$ such that
\begin{equation*}
    \tr p_t(A) = \min \{\tr p(A) : p \in \mathcal{P}_t\}.
\end{equation*}

     Let  $\theta_0 > \cdots > \theta_d$ be the distinct eigenvalues of $G$, and let $m_0,\hdots,m_d$ be their respective multiplicities. Then, for $t < d$ we have the following LP:
{\small{
    \begin{equation}\label{eq:ratio_chi_LP_walk}
        \boxed{\begin{aligned}
        \text{variables: } &x_1,\hdots,x_d \\
        \text{input: }&\text{An integer } t. \text{ The eigenvalues } \{\theta_0,\hdots,\theta_d\} \text{ and multiplicities } \{m_0,\hdots,m_d\} \\
        &\text{of a $t$-partially walk-regular graph }G.\\
        \text{output: }&\text{A vector } (x_1,\hdots,x_d) \text{ which defines a } t\text{-minor polynomial} \\
        \ \\   
        \text{minimize}  \ \ &\sum_{i=1}^d m_ix_i \\
        \text{subject to}\ \ &f[\theta_0,\hdots,\theta_m]=0, m=t+1,\hdots,d \\
        &x_i \geq 0, i =1,\hdots d
        \end{aligned}}
    \end{equation}
}}
    Note that in \eqref{eq:ratio_chi_LP_walk} $f[\theta_0,\hdots,\theta_m]$ is recursively defined via
    \begin{equation*}
        f[\theta_i,\hdots,\theta_j] := \frac{f[\theta_{i+1},\hdots,\theta_j]-f[\theta_i,\hdots,\theta_{j-1}]}{\theta_j-\theta_i},
    \end{equation*}
    with starting values $f[\theta_i] =x_i$.
    From the solution to this problem we get a vector $(x_1,\hdots,x_d)$ which defines a $t$-minor polynomial by $p_t(\theta_i) = x_i$ (with $x_0 := 1$).

\subsection{Optimality of the Ratio-type bounds for $t=2,3$} 

    Just as for the analogous bounds on $\alpha_t(G)$, Theorems \ref{thm:ratio_chi} and \ref{thm:ratio_chi_reg} actually give a class of bounds, depending on choice of $p$. While we will see that one can easily translate the closed-formula bounds from \cite[Corollary 3.3]{ACF2019} and \cite[Theorem 11]{kavi_optimal_2023} to obtain bounds on $\chi_t(G)$ for regular graphs and $t\in\{2,3\}$, (see Corollaries \ref{cor:ratio_chi_2_reg} and \ref{cor:ratio_chi_3_reg}), this same idea does not work when one drops the regularity assumption. Hence, a new approach is needed for obtaining optimal bounds from Theorem \ref{thm:ratio_chi} for small values of $t$. In the next result we find the optimal polynomial for Theorem \ref{thm:ratio_chi} when $t=2$.

    \begin{theorem}\label{thm:ratio_chi_2}
        Let $G$ be a graph with $n$ vertices and distinct eigenvalues $\theta_0 > \theta_1 > \cdots > \theta_d$ with $d \geq 2$. Let $\Delta$ be the maximum degree of $G$. Let $\theta_i$ be the largest eigenvalue such that $\theta_i \leq -\frac{\Delta}{\theta_0}$. Then,
        \begin{equation}\label{eq:ratio_chi_2}
            \chi_2(G) \geq \frac{(\theta_0 - \theta_i)(\theta_0 - \theta_{i-1})}{\Delta + \theta_i\theta_{i-1}}.
        \end{equation}
        Moreover, this is the best possible bound  that can be obtained from Theorem \ref{thm:ratio_chi} for $t=2$.
    \end{theorem}
    \begin{proof}
        Let $p \in \mathbb{R}_2[x]$ be as follows:
        \begin{equation*}
            p(x) = x^2 -(\theta_i+\theta_{i-1})x.
        \end{equation*}
        Then $p(\lambda_1) = (\theta_0-\theta_i)(\theta_0-\theta_{i-1})$ and $W(p) = \Delta$. Furthermore, $p(x)$ has its global minimum at $\frac{\theta_i+\theta_{i-1}}{2}$, hence $\theta_i$ (or analogously $\theta_{i-1}$) realize $\lambda(p)$. Thus $\lambda(p) = -\theta_i\theta_{i-1}$. Plugging these values into \eqref{eq:ratio_chi} gives the desired bound.
        
        To show optimality, we follow an analogous approach to the one by Abiad, Coutinho and Fiol \cite[Corollary 3.3]{ACF2019}. We further make use of the fact that an eigenvalue as proposed, $\theta_i \leq -\frac{\Delta}{\theta_0}$, always exists \cite{ADF2024}. 

        Take some polynomial $p(x) \in \mathbb{R}_2[x]$, and denote $p(x) = ax^2 + bx + c$. Assume, $a>0$ first. Note that $a=0$ trivially leads to $p(x)=x$ by the upcoming observations. The case $a<0$ will be investigated after. We observe that the bound in Theorem \ref{thm:ratio_chi} does not change under translation or positive scaling, and hence we can assume without loss of generality that $a=1$ and $c=0$. Then, $p(x) = x^2 + bx$ has its global minimum at $-\frac{b}{2}$. Since $p$ is a parabola, the eigenvalue $\theta_i$ for which $p(\theta_i)$ is minimal will be the one closest to $-\frac{b}{2}$. Now, we have two cases
        \begin{itemize}
            \item If $\theta_i \neq \theta_d$ then we get
            \begin{equation*}
                \frac{\theta_i+\theta_{i+1}}{2} \leq \frac{-b}{2} \leq \frac{\theta_i+\theta_{i-1}}{2},
            \end{equation*}
            in which case we know $b = -\theta_i+\tau$ for $-\theta_{i-1} \leq \tau \leq-\theta_{i+1}$.
            \item If $\theta_i = \theta_d$, then we get
            \begin{equation*}
                \frac{-b}{2} \leq \frac{\theta_d+\theta_{d-1}}{2},
            \end{equation*}
            in which case we know $b = -\theta_i+\tau$ for $-\theta_{d-1} \leq \tau$.
        \end{itemize}
        In either case, we find:
        \begin{align*}
             W(p) &= \Delta, \\
            \lambda(p) &=p(\theta_i) = \theta_i^2+( -\theta_i+\tau)\theta_i = \tau\theta_i.
        \end{align*}
        We now write the bound in \eqref{eq:ratio_chi} as a function of $\tau$
        \begin{equation*}
            \Phi(\tau) = \frac{\theta_0^2+(\tau-\theta_i)\theta_0 - \tau\theta_i}{\Delta - \tau\theta_i} = \frac{(\theta_0-\theta_i)(\theta_0 + \tau)}{\Delta - \tau\theta_i}.
        \end{equation*}
        This has derivative
        \begin{equation*}
            \Phi'(\tau) =\frac{(\Delta - \tau\theta_i)(\theta_0-\theta_i) - \theta_i(\theta_0-\theta_i)(\theta_0 + \tau)}{(\Delta - \tau\theta_i)^2} = \frac{(\Delta+\theta_i\theta_0)(\theta_0 -\theta_i)}{(\Delta - \tau\theta_i)^2}.
        \end{equation*}
        Since $\theta_0 > \theta_i$, we find that $\Phi(\tau)$ is decreasing if $\theta_i < -\frac{\Delta}{\theta_0}$, constant if $\theta_i = -\frac{\Delta}{\theta_0}$ and increasing if $\theta_i > -\frac{\Delta}{\theta_0}$.
        
        We investigate these cases individually. Recall, we want to find the best \emph{lower} bound, and thus we want to maximize $\Phi(\tau)$. 
        \begin{enumerate}
            \item[i.] If $\theta_i < -\frac{\Delta}{\theta_0}$, then $\Phi(\tau)$ is decreasing, thus we want to take $\tau$ minimal, i.e.\ $\tau = -\theta_{i-1}$. In this case $p(x) = x^2 -(\theta_i + \theta_{i-1})x$, which gives us the following expression for \eqref{eq:ratio_chi} as a function of $\theta_{i-1}$:
            \begin{equation*}
               \Phi(\theta_{i-1}) = \frac{(\theta_0 - \theta_i)(\theta_0 - \theta_{i-1})}{\Delta + \theta_i\theta_{i-1}}.
            \end{equation*}
             Now we just need to optimize the above over $\theta_i$. We do this by taking the derivative with respect to $\theta_{i-1}$:
            \begin{align*}
                \Phi'(\theta_{i-1}) &=\frac{-(\theta_0-\theta_{i-1})(\Delta + \theta_i\theta_{i-1})-\theta_{i}(\theta_0 - \theta_i)(\theta_0 - \theta_{i-1})}{(\Delta + \theta_i\theta_{i-1})^2}\\
                &= \frac{-(\theta_0-\theta_{i})(\Delta+\theta_{i}\theta_0)}{(\Delta + \theta_i\theta_{i-1})^2}.
            \end{align*}
            Since we assumed $\theta_i < -\frac{\Delta}{\theta_0}$, we find the above is increasing in $\theta_{i-1}$, and hence we must take $\theta_{i-1}$ maximal. This leads to $\theta_i$ being the largest eigenvalue less than $-\frac{\Delta}{\theta_0}$.

            \item[ii.] If $\theta_i = -\frac{\Delta}{\theta_0}$, then $\theta_{i-1} > -\frac{\Delta}{\theta_0}$. Now since $\Phi(\tau)$ is constant, we can take $\tau = -\theta_{i-1}$. Then, $\theta_{i-1}$ minimizes $p(x)$ too, hence we can just call this $\theta_i$, and what used to be $\theta_i$ we call $\theta_{i+1}$ and we are in the next case.
            
            \item[iii.] If $\theta_i > -\frac{\Delta}{\theta_0}$, then $\Phi(\tau)$ is increasing, thus we want to take $\tau$ maximal. Since $\theta_d \leq -\frac{\Delta}{\theta_0}$, we have that $\theta_i \neq \theta_d$, and hence we are in the case where $\tau$ is bounded from above. Therefore, we get $\tau = -\theta_{i+1}$ and hence $p(x) = x^2 -(\theta_i + \theta_{i+1})x$, which gives us the following expression for \eqref{eq:ratio_chi} as a function of $\theta_{i+1}$:
            \begin{equation*}
               \Phi(\theta_{i+1}) = \frac{(\theta_0 - \theta_i)(\theta_0 - \theta_{i+1})}{\Delta + \theta_i\theta_{i+1}}.
            \end{equation*}
            Now we just need to optimize the above over $\theta_i$. We do this by taking the derivative with respect to $\theta_{i+1}$:
            \begin{align*}              \Phi'(\theta_{i+1}) &= \frac{-(\theta_0-\theta_{i+1})(\Delta + \theta_i\theta_{i+1})-\theta_{i}(\theta_0 - \theta_i)(\theta_0 - \theta_{i+1})}{(\Delta + \theta_i\theta_{i+1})^2} \\
                &= \frac{-(\theta_0-\theta_{i})(\Delta+\theta_{i}\theta_0)}{(\Delta + \theta_i\theta_{i+1})^2}.
            \end{align*}
            Since we assumed $\theta_i > -\frac{\Delta}{\theta_0}$, we find the above is decreasing, and hence we must take $\theta_{i+1}$ minimal. This leads to $\theta_{i+1}$ being the largest eigenvalue less than $-\frac{\Delta}{\theta_0}$. Since $\theta_{i+1}$ also minimizes $p(x)$, we can now shift indices similar to how we did in the second case to find the bound
            \begin{equation*}
               \frac{(\theta_0 - \theta_i)(\theta_0 - \theta_{i-1})}{\Delta + \theta_i\theta_{i-1}},
            \end{equation*}
            with $\theta_i$ the largest eigenvalue less than $-\frac{\Delta}{\theta_0}$.
        \end{enumerate}

    Now we return to the case $a < 0$. Here too, we can scale and translate to only have to consider $p(x) = -x^2+bx$. Recall that to apply Theorem \ref{thm:ratio_chi}, we must have $p(\theta_0) > \lambda(p)$. I.e.\ we must have for all $i \in [2,d]$
    \begin{equation*}
        -\theta_0^2 + b\theta_0 > -\theta_i^2 + b\theta_i \iff b> \frac{\theta_0^2-\theta_i^2}{\theta_0-\theta_i} = \theta_0 + \theta_i \implies b > \theta_0 + \theta_d.
    \end{equation*}
    In this case, $\lambda(p) = p(\theta_d) = -\theta_d^2 - b\theta_d$, and $W(p) = -\Delta$, hence we get the following expression for \eqref{eq:ratio_chi} as a function of $b$:
    \begin{equation*}
        \Phi(b) = \frac{-\theta_0^2 +b\theta_0 + \theta_d^2-b\theta_d}{-\delta +\theta_d^2 - b\theta_d} = \frac{b(\theta_0-\theta_d) + \theta_d^2-\theta_0^2}{- b\theta_d - \delta +\theta_d^2}.
    \end{equation*}
    We take the derivative with respect to $b$:
    \begin{align*}
        \Phi'(b) &= \frac{(- b\theta_d - \delta +\theta_d^2)(\theta_0-\theta_d) +\theta_d (b(\theta_0-\theta_d) + \theta_d^2-\theta_0^2)}{(- b\theta_d - \delta +\theta_d^2)^2} \\
        &= \frac{(\theta_d^2-\delta)(\theta_0-\theta_d) +\theta_d (\theta_d^2-\theta_0^2)}{(- b\theta_d - \delta +\theta_d^2)^2}\\
        &= \frac{(\theta_0\theta_d+\delta)(\theta_d-\theta_0)}{(- b\theta_d - \delta +\theta_d^2)^2}.
    \end{align*}
    Clearly, $\theta_d - \theta_0 < 0$. Furthermore, since $\theta_d \leq -1$, and $\theta_0 \geq \delta$ we have $\theta_0\theta_d+\delta \leq 0$, and hence the above derivative in total is nonnegative. This means that we find the best bound when $b \rightarrow \infty$, which gives $\lim_{b \rightarrow \infty} \Phi(b) = \frac{\theta_d-\theta_0}{\theta_d}$. Now it is left to show that this bound is always worse than \eqref{eq:ratio_chi_2}. To do this, note that the derivative of $\frac{\theta_d-\theta_0}{\theta_d}$ with respect to $\theta_d$, is $\frac{\theta_0}{\theta^2_d}$, this means it is increasing. Hence, the best bound will be when $\theta_d = -\frac{\Delta}{\theta_0}$, i.e.\
    \begin{equation*}
        \frac{\theta_0 - \theta_d}{-\theta_d} \leq \frac{\theta_0 + \frac{\Delta}{\theta_0}}{\frac{\Delta}{\theta_0}} = 1 + \frac{\theta^2_0}{\Delta}.
    \end{equation*}
    On the other hand, similar to before, we can take the derivative of \eqref{eq:ratio_chi_2} with respect to $\theta_i$ and observe that \eqref{eq:ratio_chi_2} is minimal when $\theta_i = -\frac{\Delta}{\theta_0}$. This means
    \begin{align*}
        \frac{(\theta_0 - \theta_i)(\theta_0 - \theta_{i-1})}{\Delta + \theta_i\theta_{i-1}} &\geq \frac{(\theta_0 +\frac{\Delta}{\theta_0})(\theta_0 - \theta_{i-1})}{\Delta -\frac{\Delta}{\theta_0} \theta_{i-1}} = \frac{\theta_0^2-\theta_0\theta_{i-1} + \Delta -\frac{\Delta}{\theta_0}\theta_{i-1}}{\Delta -\frac{\Delta}{\theta_0} \theta_{i-1}} \\
        &= 1+\frac{\theta_0^2-\theta_0\theta_{i-1}}{\Delta -\frac{\Delta}{\theta_0} \theta_{i-1}} = 1 + \frac{\theta_0(\theta_0-\theta_{i-1})}{\frac{\Delta}{\theta_0}(\theta_0-\theta_{i-1})} = 1 + \frac{\theta^2_0}{\Delta}.
    \end{align*}
Therefore, \eqref{eq:ratio_chi_2} is always at least as good as $\frac{\theta_d-\theta_0}{\theta_d}$. Thus, we can conclude the best bound that can be obtained from Theorem \ref{thm:ratio_chi} is
        \begin{align*}
           & \chi_2(G) \geq \frac{(\theta_0 - \theta_i)(\theta_0 - \theta_{i-1})}{\Delta + \theta_i\theta_{i-1}}.
       \qedhere \end{align*}
    \end{proof}

    Since we were able to provide an optimal polynomial for Theorem \ref{thm:ratio_chi} when $t=2$, which holds for general graphs, the next question would be if we can do the same for $t=3$, especially considering that we have the approach of \cite[Theorem 11]{kavi_optimal_2023} to work with as a base. Unfortunately, this is not as simple. To prove a generalization of \cite[Corollary 3.3]{ACF2019}, we mostly just have to swap some $\theta_0$'s for $\Delta$'s in the proof of \cite[Corollary 3.3]{ACF2019} and then deal with a few technicalities. On the other hand, in the proof of \cite[Theorem 11]{kavi_optimal_2023}, the authors take advantage of the fact that $W(p) = \max_{u\in V} \{(A^3)_{uu}\} +bk$ (for $p(x) = x^3 + bx^2 + cx$). This does not generalize nicely to non-regular graphs. Instead of $W(p)$ being linear in $b$, it is the maximum of $|V|$ linear functions in $b$, which makes several steps in the proof not work. Thus, finding the optimal polynomial for $t=3$ in Theorem \ref{thm:ratio_chi} is not a trivial task. 

   On the other hand, and as mentioned earlier, if one assumes graph regularity, it is straightforward to determine the optimal choices of $p$ for $t\in\{2,3\}$ in Theorem \ref{thm:ratio_chi_reg}. This is illustrated in the following two corollaries.
      
     \begin{corollary}\label{cor:ratio_chi_2_reg}
        Let $G$ be a $k$-regular graph with $n$ vertices with distinct eigenvalues $k=\theta_0 > \theta_1 > \cdots > \theta_d$ with $d \geq 2$. Let $\theta_i$ be the largest eigenvalue such that $\theta_i \leq -1$. Then 
        \begin{equation*}
            \chi_2(G) \geq \frac{n}{\left\lfloor n\frac{\theta_0 + \theta_i\theta_{i-1}}{(\theta_0 - \theta_i)(\theta_0 - \theta_{i-1})}\right\rfloor}.
        \end{equation*}
        Moreover, no better bound can be obtained from Theorem \ref{thm:ratio_chi_reg}.
    \end{corollary}
    \begin{proof}
       Apply  \cite[Corollary 3.3]{ACF2019} to $G$, followed by \eqref{lem:alpha_chi_relation} to obtain the desired bound. Now we must show optimality. Assume there is some $p$ which, when used to obtain a bound from Theorem \ref{thm:ratio_chi_reg}, gives a tighter bound on $\chi_2(G)$. Then, this $p$ will also result in a tighter bound on $\alpha_2(G)$ than the one from \cite[Corollary 3.3]{ACF2019}. This contradicts the optimality assertion from \cite[Corollary 3.3]{ACF2019}. Thus, no such $p$ can exist, and  no better bound can be obtained from Theorem \ref{thm:ratio_chi_reg}.
    \end{proof}

    \begin{corollary}\label{cor:ratio_chi_3_reg}
         Let $G$ be a $k$-regular graph with $n$ vertices with adjacency matrix $A$ and distinct eigenvalues $k = \theta_0 > \theta_1 > \cdots > \theta_d$, with $d \geq 3$. Let $\theta_s$ be the largest eigenvalue such that $\theta_s \leq - \frac{\theta^2_0 + \theta_0\theta_d-\Delta_3}{\theta_0 (\theta_d + 1)}$, where $\Delta_3 = \max_{u\in V} \{(A^3)_{uu}\}$. Then, the distance-$3$ chromatic number satisfies:
        \begin{equation*}
            \chi_3(G) \geq \frac{n}{\left\lfloor n\frac{\Delta_3 - \theta_0(\theta_s + \theta_{s-1} + \theta_d) - \theta_s\theta_{s-1} \theta_d}{(\theta_0 - \theta_s)(\theta_0 - \theta_{s-1})(\theta_0 - \theta_d)}\right\rfloor}.
        \end{equation*}
        Moreover, no better bound can be obtained from Theorem \ref{thm:ratio_chi_reg}.
    \end{corollary}
    \begin{proof}
        Apply \cite[Theorem 11]{kavi_optimal_2023} to $G$, followed by \eqref{lem:alpha_chi_relation} to obtain the desired bound. The proof of optimality is analogous to that of Corollary \ref{cor:ratio_chi_2_reg}.
    \end{proof}

\section{Bounding the distance chromatic number of the hypercube graph} \label{sec:chit_Qn}

    In this section we investigate the distance chromatic number of  hypercube graphs. The \emph{hypercube graph} of dimension $n$ (or simply $n$-cube), denoted $Q_n$, is the $n$-fold Cartesian product of $K_2$, the complete graph on two vertices. Hypercube graphs have a strong connection to coding theory, as each vertex of $Q_n$ can be seen as a vector in $\mathcal A_2^n=\{0,1\}^n$.
    Recall that one can define the \emph{Hamming distance} $\mathrm{d}_{\mathrm{H}}(u,v)$ between two vectors $u,v\in \mathcal A_2^n$ as the number of coordinate in which they differ.
    In particular, in the  hypercube graph $Q_n$, there is an edge connecting two vertices if the Hamming distance of their vector representation is $1$. It is known that the $t$-independence number of $Q_n$ equals 
        $$A_2(n,t+1):=\max\{|C| \,:\, C\subseteq \mathcal A_2, \mathrm{d}_{\mathrm{H}}(u,v)\ge t+1, \,\forall\, u,v \in C, u\neq v\}, $$ one of the most interesting and studied quantities in  coding theory. 
    In fact, this connection   holds in a more general setting, see e.g. Abiad, Khramova and Ravagnani \cite[Corollary~16]{abiad_eigenvalue_2023}.

The distance-$t$ chromatic number of the hypercube graph has been of interest in past research for its relation to $A_2(n,t+1)$. However, not even the distance-$2$ chromatic number is known for all $Q_n$. The problem of determining $\chi_2(Q_n)$ was first studied in 1997 by Wan \cite{wan_near-optimal_1997}, motivated by a connection to optical networks, who provided a bound on the distance-$2$ chromatic number of $Q_n$ and proposed the problem of finding bounds for general $t$. Wan \cite{wan_near-optimal_1997} conjectured that his upper bound was tight for all $n$, that is, that $\chi_2(Q_n) = 2^{\lceil\log_2(n+1)\rceil}$. However, this conjecture turned out to be false, as $13 \leq \chi_2(Q_8) \leq 14$ \cite[Section 9.1]{ziegler_coloring_2001}.
Kim, Du and Pardalos \cite{kim_coloring_2000} showed a similar result for $\chi_3(Q_n)$, and they also provided a more crude bound for the general distance-$t$ chromatic number. In 2002, Ngo, Du and Graham \cite{ngo_new_2002} improved upon and extended these previous bounds for $\chi_t(Q_n)$ by using a version of the Johnson bound.
    \begin{theorem}[{\cite[Theorem 1]{ngo_new_2002}}]\label{thm:Qn_chit_ngobound}
        Let $Q_n$ be the hypercube graph of dimension $n$. Let $s:= \lfloor \frac{t}{2}\rfloor$ and denote $(\binom{n}{t}) := \sum_{i=0}^t \binom{n}{i}$. Then the distance-$t$ chromatic number of $Q_n$ satisfies
        {\footnotesize{
        \begin{align*}
            \left(\binom{n}{s}\right) + \frac{1}{\left\lfloor \frac{n}{s+1}\right\rfloor} \binom{n}{s} \left(\frac{n-s}{s+1} - \left\lfloor \frac{n-s}{s+1}\right\rfloor\right) & \leq \chi_t(Q_n) \leq 2^{\left\lfloor \log_2 \left(\binom{n-1}{t-1}\right)\right\rfloor+1} \  &\text{ if } t \text{ even,}\\
           2\left(\left(\binom{n-1}{s}\right) + \frac{1}{\left\lfloor \frac{n-1}{s+1}\right\rfloor} \binom{n-1}{s} \left(\frac{n-1-s}{s+1} - \left\lfloor \frac{n-1-s}{s+1}\right\rfloor\right)\right) &\leq \chi_t(Q_n) \leq 2^{\left\lfloor \log_2 \left(\binom{n-2}{t-2}\right)\right\rfloor+2} \ & \text{ if } t \text{ odd.}
        \end{align*}
        }}
    \end{theorem}

    Of particular interest to us are the lower bounds from Theorem \ref{thm:Qn_chit_ngobound} for $t\in\{2,3\}$, since in Section \ref{sec:chit_Qn} we will provide alternative spectral proofs for them. The existing proofs, similarly as ours, rely on using existing bounds on $\alpha_t(G)/A_q(n,d)$.
    
    \begin{corollary}[{\cite{ngo_new_2002}}]\label{cor:Qn_chi2_ngobound}
        Let $Q_n$ be the hypercube graph of dimension $n$, with $n\geq 2$. Then, the distance-$2$ chromatic number satisfies
        \begin{equation*}
            \chi_2(Q_n) \geq \begin{cases}
                 n+2 &\text{if } n \text{ even,} \\
                 n+1 &\text{if } n \text{ odd.}
            \end{cases} 
        \end{equation*}
    \end{corollary}

    In 2008, Jamison, Matthews and Villalpando \cite{jamison_acyclic_2006} independently proved the same lower bound as in Corollary \ref{cor:Qn_chi2_ngobound}.

    \begin{corollary}[{\cite{ngo_new_2002}}]\label{cor:Qn_chi3_ngobound}
        Let $Q_n$ be the hypercube graph of dimension $n$, with $n \geq 3$. Then the distance-$3$ chromatic number satisfies
        \begin{equation*}
            \chi_3(Q_n) \geq \begin{cases}
                2n &\text{ if } n \text{ even,} \\
                2(n + 1) &\text{ if } n \text{ odd.}
            \end{cases}
        \end{equation*}
    \end{corollary}

    For certain specific cases, exact values of $\chi_t(Q_n)$ are known, as the following result shows. 
    
    \begin{theorem}[{\cite[Theorem 2.15]{francis_b-coloring_2017}}]\label{thm:equalitychihypercube}
        Let $Q_n$ be the hypercube graph of dimension $n$, with $n \geq 2$. If $\frac{2(n-1)}{3}\leq t \leq n-1$, then $\chi_t(Q_n) = 2^{n-1}$.
    \end{theorem}

    In particular, while finding the exact value of $\chi_2(Q_n)$ for small $n$ has received much attention (see e.g. \cite{kokkala_chromatic_2017, lauri2016square, jensen_coloring_1995, ziegler_coloring_2001}), it continues to be an open problem even for $n$ as small as $9$.

    Next, we use the bounds from Corollaries \ref{cor:ratio_chi_2_reg} and \ref{cor:ratio_chi_3_reg} on hypercube graphs in order to obtain alternative spectral proofs for Corollaries \ref{cor:Qn_chi2_ngobound} and \ref{cor:Qn_chi3_ngobound}. Additionally, we provide new bounds on $\chi_4(Q_n)$ and $\chi_5(Q_n)$ in Corollary \ref{cor:Qn_chit_ratiobound}. Before we present the new proof, we need the following preliminary result.

        \begin{theorem}[{\cite[Theorem 9.2.1]{brouwer_distance-regular_1989}}]\label{thm:Qn_spectrum}
        The hypercube graph $Q_n$ has adjacency eigenvalues $\theta_l = (n-2l)$ with multiplicities $m_l = \binom{n}{l}$ for $l \in [0,n]$.
    \end{theorem}

    \begin{proof}[Alternative proof of Corollary \ref{cor:Qn_chi2_ngobound}]
        We apply Corollary \ref{cor:ratio_chi_2_reg} to $Q_n$. For this we use Theorem \ref{thm:Qn_spectrum}, which tells us that the adjacency eigenvalues of $Q_n$ are $\{n,n-2,\hdots,-n\}$. In particular, $\theta_0 = n$, $\theta_d = -n$. For $\theta_i$, we must split into two cases, $n$ even or $n$ odd. For $n$ even we have $\theta_i = -2$, $\theta_{i-1} = 0$ and for $n$ odd we have $\theta_i = -1$, $\theta_{i-1} = 1$. This gives
        \begin{equation*}
            \chi_2(Q_n) \geq \frac{(\theta_0 - \theta_i)(\theta_0 - \theta_{i-1})}{\theta_0 + \theta_i\theta_{i-1}} = \begin{cases}
                \frac{(n+2)n}{n}  = n+2 &\text{if } n \text{ even,}\\
                \frac{(n+1)(n-1)}{n-1}  = n+1 &\text{if } n \text{ odd,}\\
            \end{cases}
        \end{equation*}
        as desired.
    \end{proof}
   
    \begin{proof}[Alternative proof of Corollary \ref{cor:Qn_chi3_ngobound}]
         We apply Corollary \ref{cor:ratio_chi_3_reg} to $Q_n$. For this, we use Theorem \ref{thm:Qn_spectrum}, which gives us that the adjacency eigenvalues of $Q_n$ are $\{n,n-2,\hdots,-n\}$. In particular, $\theta_0 = n$, $\theta_d = -n$. Note that none of the hypercube graphs have a closed walk of length $3$, this means all diagonal entries of $A^3$ are 0, and thus $\Delta_3 = 0$. Recall $\theta_s$ is the largest eigenvalue such that $\theta_s \leq -\frac{\theta_0^2 + \theta_0\theta_d-\Delta_3}{\theta_0(\theta_d+1)}$, in this case we have $-\frac{\theta_0^2 + \theta_0\theta_d-\Delta_3}{\theta_0(\theta_d+1)} = -\frac{n^2-n^2-0}{n(1-n)} = 0$. We must distinguish the cases $n$ even and $n$ odd. If $n$ is even, then $\theta_s = 0$ and $\theta_{s-1} = 2$, if $n$ is odd then $\theta_s = -1$ and $\theta_{s-1} = 1$. We can now plug all these values into the bound from Corollary \ref{cor:ratio_chi_3_reg} to obtain
         {\small{
        \begin{align*}
            \chi_3(Q_n) &\geq \frac{ (n - 0)(n - 2)(n + n)}{0-n(0 + 2 -n) + 0\cdot 2n} = \frac{2n^2(n-2)}{n(n-2)} = 2n &\text{ if } n \text{ even,} \\
            \chi_3(Q_n) &\geq \frac{(n +1 )(n - 1)(n + n)}{0-n(-1 + 1 -n) -1\cdot 1 \cdot n} = \frac{2n(n+1)(n-1)}{n(n-1)} = 2(n+1)  &\text{ if } n \text{ odd,}
        \end{align*}
        }}
        as desired.
    \end{proof}

    While we do not have a closed expression for the best possible Ratio-type bound for $t \geq 4$ like we do for $t=2,3$, we can instead use LP \eqref{eq:ratio_chi_LP_walk} to obtain bounds for specific graphs. In this particular case of hypercube graphs $Q_n$, the polynomials obtained from LP \eqref{eq:ratio_chi_LP_walk} for small $n$ follow a clear pattern. We are able to extend this pattern to the general case, at least for $t=4,5$, and obtain a closed expression bound in Corollary \ref{cor:Qn_chit_ratiobound}. As we will see in Tables \ref{tab:Qn_t4} and \ref{tab:Qn_t5}, computational experiments suggest that the new bound from Corollary \ref{cor:Qn_chit_ratiobound} outperforms the previously known bound from Theorem \ref{thm:Qn_chit_ngobound} in certain cases. Before stating this bound, we first need to characterize the diagonal entries of the powers of the adjacency matrix of the hypercube graphs.

    \begin{lemma}\label{lem:Qn_walks}
        Let $Q_n$ be the hypercube graph of dimension $n$ with adjacency matrix $A$, let $i \in \mathbb{N}$. Then
        \begin{equation*}
            (A^i)_{uu} = \frac{1}{2^n}\sum_{i=0}^n \binom{n}{i} (n-2i)^i.
        \end{equation*}
    \end{lemma}
    \begin{proof}
        We know $\tr A^i = \sum_{i=1}^n \lambda_i^i = \sum_{i=0}^n \binom{n}{i} (n-2j)^i$. By the walk-regularity of $Q_n$, we have $(A^i)_{uu} = \frac{1}{2^n} \tr A^i$, which gives us the desired equality.
    \end{proof}

    \begin{corollary}\label{cor:Qn_chit_ratiobound}
        Let $t\in \{4,5\}$ and let $Q_n$ be the hypercube graph of dimension $n$, with $n \geq t$. Let 
        \begin{equation*}
            m = \left\lfloor\frac{\sqrt{n+3-t}+n+3-t}{2} - \left\lceil\frac{n+4-t}{2}\right\rceil\right\rfloor.
        \end{equation*}
        Define
        \begin{equation*}
            R_{t,m} = \begin{cases}
                \{-(2m+4), -(2m+2), 2m, (2m+2)\} &\text{ if } t=4 \text{ and } n \text{ even,} \\
                \{-(2m+5), -(2m+3), (2m+1), (2m+3)\} &\text{ if } t=4 \text{ and } n \text{ odd,} \\
                \{-n, -(2m+4), -(2m+2), (2m+2), (2m+4)\} &\text{ if }t=5 \text{ and } n \text{ even,} \\
                \{-n, -(2m+3), -(2m+1), (2m+1), (2m+3)\} &\text{ if } t=5 \text{ and } n \text{ odd,}
                \end{cases}
        \end{equation*}
        and let $a_i = \frac{1}{2^n}\sum_{i=0}^n \binom{n}{i} (n-2i)^i$. Consider the polynomial
        \begin{equation*}
            p_{t,m}(x) = \prod_{r \in R_{t,m}} (x-r), 
        \end{equation*} 
        which has coefficients:
        \begin{equation*}
            b_{t,m,i} = (-1)^{t+i} \sum_{\substack{S \subset R_{t,m} \\ \abs{S} = (t-i)}} \prod_{r \in S} r.
        \end{equation*}
        Then, the distance-$t$ chromatic number satisfies
        \begin{equation*}
            \chi_t(Q_n) \geq \frac{p_{t,m}(n)}{\sum_{i=0}^t a_ib_{t,m,i}}.
        \end{equation*}
    \end{corollary}
    \begin{proof}
        We will use the polynomial $p_{t,m}$ in Theorem \ref{thm:ratio_chi_reg}. By Lemma \ref{lem:Qn_walks}, we have 
        \begin{equation*}
            W(p) = \sum_{i=0}^t (A^t)_{uu} b_{t,m,i} = \sum_{i=0}^t a_ib_{t,m,i}.
        \end{equation*}

        Now if we denote by $x_1<\cdots<x_t$ the roots of $p_{t,m}$, we observe the following: If $t=4$, then $p_{t,m}(x) < 0$ if and only if $x_1<x<x_2$ or $x_3<x<x_4$. Furthermore, we observe that all the $x_i$'s are eigenvalues of $Q_n$ and that they are chosen such that $\abs{x_1-x_2}=\abs{x_3-x_4}=2$. In particular, since the eigenvalues of $Q_n$ are all exactly $2$ apart, we find that $Q_n$ has no eigenvalues which lie between $x_1$ and $x_2$ or between $x_3$ and $x_4$. Hence, for all eigenvalues we have $p_{t,m}(\lambda_i) \geq 0$, with equality being achieved for the aforementioned roots. So $\lambda(p) = 0$. For the case $t=5$, since $p_{t,m}$ is of degree $5$ we have $p_{t,m}(x) < 0$ if and only if $x<x_1$ or $x_2<x<x_3$ or $x_4<x<x_5$ instead. Note that $x_1$ is the smallest eigenvalues of $Q_n$, hence no eigenvalues fall in the first range. For the latter two ranges, the same argument as in the case $t=4$ applies as to why no eigenvalues of $Q_n$ lie in there. Thus here too $\lambda(p) = 0$.
        Plugging these values into \eqref{eq:ratio_chi} gives the desired bound.
    \end{proof}

    As mentioned, the polynomials chosen in Corollary \ref{cor:Qn_chit_ratiobound} are a generalization of polynomials obtained from LP \eqref{eq:ratio_chi_LP_walk} for small values of $n$. These polynomials have a very specific form, with roots occurring in pairs at eigenvalues of $Q_n$ such that $\lambda(p)=0$ (in the case $t=5$ we have an additional root at $-n$ to ensure the same effect). How far out these pairs are from $0$ increases as $n$ increases. 
    
    Running LP \eqref{eq:ratio_chi_LP_walk} for larger $t$ suggests polynomials with a similar pattern are optimal for $t>5$ as well, though it gets increasingly complex. For $t\in\{6,7\}$ the roots are still manageable. As $n$ increases, one pair of roots remains stationary around $0$, whereas the other two pairs move outward like before, with an additional root at $-n$ for $t=7$ analogously to the case for $t=5$. However, determining simple closed formula for $m$ in these cases (i.e. the ``rate'' at which roots move away from $0$) is not as simple. On the other hand, we want to highlight the fact that, for $t=6$, the values of $n$ for which $m$ increases line up with a so-called \emph{crystal ball sequence} (in particular, the crystal ball sequence corresponding to the truncated square tiling), perhaps suggesting a deeper connection there. Such a crystal ball sequence is obtained by taking the partial sums of the \emph{coordination sequence} of a tiling; see \cite{OK1995} for further details. For even larger $t \geq 8$, LP \eqref{eq:ratio_chi_LP_walk} returns polynomials whose multiple pairs of roots on each side of $0$ move outward at different rates, even further complicating the scenario.

\subsection*{Computational results}

In Tables \ref{tab:Qn_t2}, \ref{tab:Qn_t3}, \ref{tab:Qn_t4}, \ref{tab:Qn_t5} from the Appendix, the performance of the new bounds from Section \ref{sec:chit_Qn} are presented. These were computed using \textsc{SageMath}. 

In Tables \ref{tab:Qn_t2} and \ref{tab:Qn_t3} we compare the performance of the previously known bounds by Ngo, Du and Graham (Corollaries \ref{cor:Qn_chi2_ngobound} and \ref{cor:Qn_chi3_ngobound}) to known exact values of $\chi_t(Q_n)$. Wherever there is no citation present in the ``$\chi_t(Q_n)$'' column, the value was obtained computationally. An entry is replaced by ``time'' if it took more than 30 minutes to compute it. In Tables \ref{tab:Qn_t4} and \ref{tab:Qn_t5} we compare the performance of the previously known bound by Ngo, Du and Graham  (Theorem \ref{thm:Qn_chit_ngobound}) to our new spectral bound (Corollary \ref{cor:Qn_chit_ratiobound}). The entries in the ``$\chi_t(Q_n)$'' column are obtained computationally, and replaced by ``time'' if it took more than 30 minutes to compute.

Notably, our new bound from Corollary \ref{cor:Qn_chit_ratiobound} sometimes outperforms the known bound due to Ngo, Du and Graham from Theorem \ref{thm:Qn_chit_ngobound}, for $t\in\{4,5\}$. 
    
\section{The Lee metric and the Lee graph}\label{sec:Gnq_properties}

In this section we focus on a graph theoretic interpretation of error-correcting codes endowed with the Lee distance. Let $q$ be a positive integer greater than $1$, and let us consider the $q$-ary alphabet $\mathcal A_q=\{0,1,\ldots,q-1\}$. Let $n$ be a positive integer. We can endow the set $\mathcal A_q^n$ with the \emph{Lee distance}, which is defined, for $b=(b_1,\ldots,b_n), c=(c_1,\ldots,c_n)$ as
    $$\mathrm{d}_{\mathrm{L}}(b,c)=\sum_{i=1}^n \min\{|b_i-c_i|,q-|b_i-c_i|\}.$$
    An $(n,M,d)_q$-Lee code is defined as a nonempty subset $C\subseteq A_q^n$ of size $|C|=M\geq 2$, where 
    $$d=\mathrm{d}_{\mathrm{L}}(C):=\min\{\mathrm{d}_{\mathrm{L}}(b,c) \,:\, b,c \in C, b \neq c\}.$$
    The integer $n$ is called the \emph{(code) length} of $C$.

  \medskip

One of the central research questions in the theory of codes endowed with the Lee metric is finding the largest size $q$-ary code of minimum Lee distance $d$ and length $n$. Formally, given $q,n$ and $d$, the problem is finding the quantity 
$$A^{\mathrm{L}}_q(n,d):=\max\{ M \,:\, \mbox{ there exists an} (n,M,d)_q\mbox{-Lee code}\}.$$

Just as the $t$-independence number of hypercube graphs is connected to the maximum size of binary codes with minimum Hamming distance $t+1$, we can relate $A^L_q(n,t+1)$ to the $t$-independence number of a family of graphs as well. 
To this end, Kim and Kim \cite{kim_2-distance_2011} defined the \emph{Lee graph} $G(n,q):= C_q \square \cdots \square C_q$, that is, the $n$-fold Cartesian product of the $q$-cycle with itself. Note that the Lee graphs are a generalization of the hypercube graphs, since $Q_n=G(n,2)$. The Lee graph $G(n,q)$ has $q^n$ vertices, and such vertices can be represented as elements of $\mathcal A_q^n$. Figure \ref{fig:Gnq_example} shows an example of what $G(n,q)$ looks like.
    
         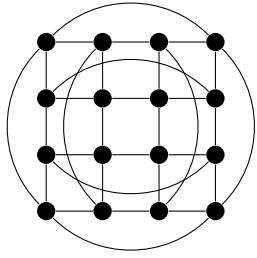
\begin{figure}[!htp]
         \centering
         \tiny
         \begin{tikzpicture}[scale=0.75]
             \node[fill,circle] (u1) at (0,0){};
             \node[fill,circle] (u2) at (0,1){};
             \node[fill,circle] (u3) at (0,2){};
             \node[fill,circle] (u4) at (0,3){};
             \node[fill,circle] (u5) at (1,0){};
             \node[fill,circle] (u6) at (1,1){};
             \node[fill,circle] (u7) at (1,2){};
             \node[fill,circle] (u8) at (1,3){};
             \node[fill,circle] (u9) at (2,0){};
             \node[fill,circle] (u10) at (2,1){};
             \node[fill,circle] (u11) at (2,2){};
             \node[fill,circle] (u12) at (2,3){};
             \node[fill,circle] (u13) at (3,0){};
             \node[fill,circle] (u14) at (3,1){};
             \node[fill,circle] (u15) at (3,2){};
             \node[fill,circle] (u16) at (3,3){};
             \draw (u1)--(u2)--(u3)--(u4);
             \draw (u5)--(u6)--(u7)--(u8);
             \draw (u9)--(u10)--(u11)--(u12);
             \draw (u13)--(u14)--(u15)--(u16);
             \draw (u1)--(u5)--(u9)--(u13);
             \draw (u2)--(u6)--(u10)--(u14);
             \draw (u3)--(u7)--(u11)--(u15);
             \draw (u4)--(u8)--(u12)--(u16);
             \draw (u1) to[bend left=45] (u4);
             \draw (u5) to[bend left=45] (u8);
             \draw (u9) to[bend right=45] (u12);
             \draw (u13) to[bend right=45] (u16);
             \draw (u1) to[bend right=45] (u13);
             \draw (u2) to[bend right=45] (u14);
             \draw (u3) to[bend left=45] (u15);
             \draw (u4) to[bend left=45] (u16);
         \end{tikzpicture}
         \caption{The Lee graph $G(2,4)$.}
         \label{fig:Gnq_example}
     \end{figure}
     
     The Lee graph $G(n,q)$ represents the Lee metric in $\mathcal A_q^n$ in the same way as the hypercube graph $Q_n$ represents the Hamming metric in $\mathcal A_2^n$. This will be made clear in the following results. For a graph $G = (V,E)$, a simple measure of the distance between two vertices $v,u \in V$ is the \emph{geodesic distance} between two vertices, which is the length in terms of the number of edges of the shortest path between the vertices, denoted by $d_G(u,v)$. For two vertices that are not connected in a graph, the geodesic distance is defined as infinite. 
     
     In Section \ref{sec:chit_Gnq} we will apply the eigenvalue bounds on $\chi_t(G)$ from Section \ref{sec:spec_chit} to the Lee graph, and for that we need the following preliminary results.

    \begin{lemma}\label{lem:Lee_Gnq_distance}
         Let $G(n,q)$ be a Lee graph, and consider the vertices as vectors in $\mathcal A_q^n$. Then the geodesic distance in $G(n,q)$ coincides with the Lee metric distance. 
    \end{lemma}
    \begin{proof}
        Let $u,v \in V(G(n,q))$ with $u = (u_1,\hdots,u_n)$, $v = (v_1,\hdots,v_n)$. Denote $d_i = \min\{q-|u_i-v_i|,|u_i-v_i|\}$, then clearly $d_L(u,v) = \sum_{i=1}^n d_i$. Note that two vertices are adjacent in $G(n,q)$ if and only if exactly one of the $d_i$ is $1$ and all others are $0$. Hence, any path between $u$ and $v$ must be of length at least $d_L(u,v)$. As we can always reduce the sum of the $d_i$ by $1$ at each step, we can find a path of length exactly $d_L(u,v)$. This means $d_{G(n,q)}(u,v) = d_L(u,v)$. This holds for arbitrary $u$ and $v$, meaning these distances are indeed equivalent.
    \end{proof}
    
    \begin{corollary}\label{cor:indep_Gnq_equiv}
        Let $G(n,q)$ be a Lee graph, and consider the vertices as vectors in $\mathcal A_q^n$. Then,  $U$ is a distance-$t$ independent set of size $M$ in $G(n,q)$ if and only if it is an $(n,M,t+1)_q$-Lee code.
    \end{corollary}

The next result will be useful later on to provide some new nonexistence conditions on perfect Lee codes.

    \begin{corollary}\label{cor:alpha_AL_equiv}
        Let $G(n,q)$ be a Lee graph. Then, $\alpha_t(G(n,q)) = A^L_q(n,t+1)$.
    \end{corollary}
    \begin{proof}
        This follows directly from Lemma \ref{cor:indep_Gnq_equiv}.
    \end{proof}

    The implications of Corollary \ref{cor:alpha_AL_equiv} will be explored in Section \ref{sec:chit_Gnq} by means of the new spectral bounds on $\chi_t(G)$ from Section \ref{sec:spec_chit}.

\subsection{Combinatorial properties of the Lee graph}

Next we investigate some combinatorial and spectral properties of Lee graphs that will be needed to later on  apply the corresponding eigenvalue bounds on such graphs. Indeed, in order to apply the spectral bounds from Section \ref{sec:spec_chit} to $G(n,q)$, we will have to use the LP given in \eqref{eq:ratio_chi_LP_walk}. For this to be possible, it must hold that $G(n,q)$ is $t$-partially walk-regular for a given $t$. In fact, we can easily show that $G(n,q)$ is \emph{walk-regular}, that is, it is \mbox{$l$-partially} walk-regular for any positive integer~$l$. 
    
    \begin{lemma}\label{lemma:wr}
        The Lee graph $G(n,q)$ is walk-regular.
    \end{lemma}
    \begin{proof}
        We can interpret $\mathcal A_q$ as the ring $\mathbb Z/q\mathbb Z$, and consider the vertices of $G(n,q)$ as elements of the module $(\mathbb Z/q\mathbb Z)^n$
        Consider $u_0=(0,\ldots,0)$ and take an arbitrary $v_0 \in (\mathbb Z/q\mathbb Z)^n$.. We will now show for any arbitrary walk from $u_0$ to itself, there is a corresponding walk from $v_0$ to itself. Consider a closed walk from $u_0$ to itself and consider the vertices in this path: $(u_0,u_1,\hdots,u_l,u_0)$. Then, we can construct the following walk from $v_0$ to itself: $(v_0,v_0+u_1,\hdots,v_0+u_l,v_0)$. One can easily verify that any two distinct walks from $u_0$ to itself give rise to two distinct walks from $v_0$ to itself.
    \end{proof}

    Lemma \ref{lemma:wr} is crucial in order to use the Ratio-type eigenvalue bound on $\chi_t(G)$ and its LP implementation \eqref{eq:ratio_chi_LP_walk} later on, since our optimization method requires the graph to be partially walk-regular.

    A graph is said to be \emph{distance-regular} if it is regular and for any two vertices $v$ and $w$, the number of vertices at distance $j$ from $v$ and at distance $k$ from $w$ depends only on $j$, $k$, and the distance between $v$ and $w$. While $G(n,q)$ is always walk-regular, it is only distance-regular for very small cases (for example $G(2,5)$ is already not distance-regular). This means that Delsarte's LP bound \cite{delsarteLP} cannot be applied in a straightforward manner, while our new eigenvalue bounds provide an easy method to be used. Using an association scheme, Astola and Tabus \cite{astola_bounds_2013} were able to overcome this difficulty and obtain bounds $A^L_q(n,d)$ using the LP method.

\subsection{Spectral properties of the Lee graph}\label{sec:Gnq_spec_prelims}
 
 In order to apply the spectral bounds on $\chi_t(G)$ from Section \ref{sec:spec_chit}, we must first find the adjacency spectrum of $G(n,q)$. 
 Denote by $\zeta_q$ the primitive $q$th root of unity $\zeta_q := e^{\frac{2\pi i}{q}}$.

    \begin{lemma}[{\cite[Section 1.4.3]{brouwer_spectra_2012}}]\label{lem:Cq_spectrum}
        The cycle graph $C_q$ has adjacency spectrum 
      \begin{equation*}
            \left\{\zeta_q^l + \zeta_q^{q-l} : l \in [1,q] \right\}.
        \end{equation*}
    \end{lemma}

    \begin{theorem}[{\cite[Pages 328-329]{barik_spectra_2018}}]\label{thm:prod_spectrum}
        Let $G_1,G_2$ be graphs with respective eigenvalues \linebreak $\{\lambda_1,\hdots,\lambda_n\}$ and $\{\mu_1,\hdots,\mu_m\}$, then $G_1 \square G_2$ has eigenvalues $\{\lambda_i+\mu_j : 1 \leq i \leq n, 1 \leq j \leq m\}$.
    \end{theorem}

    \begin{corollary}\label{cor:Gnq_spectrum}
        The adjacency spectrum of the Lee graph  $G(n,q)$ is
        \begin{equation*}
            \left\{\sum_{i=1}^n \zeta_q^{l_i} + \zeta_q^{q-l_i} : l_1,\ldots,l_n \in [1,q] \right\}.
        \end{equation*}
    \end{corollary}


    Unlike for the hypercube graphs $Q_n$, the spectrum of $G(n,q)$ is quite hard to get a grip on. As seen in Lemma \ref{cor:Gnq_spectrum}, the spectrum of $G(n,q)$ consists of sums of cosines (note $\zeta_q^l + \zeta_q^{q-l} = 2\cos(\frac{2l\pi}{q})$) making it quite difficult to theoretically determine $\theta_i$ and $\theta_s$ for Corollary \ref{cor:ratio_chi_2_reg} and Corollary \ref{cor:ratio_chi_3_reg}, respectively. However, with some careful analysis we are able to still determine a closed expression of our bound on $\chi_2(G(n,q))$. Eigenvalues of $G(n,q)$ close or equal to $-1$ play a central role in this approach. Lemma \ref{lem:Gnq_close_eigs} establishes the existence of eigenvalues close to $-1$ in most cases, Lemma \ref{lem:Gnq_-1_eig} fully characterizes when $G(n,q)$ has $-1$ as an eigenvalue.
    
    \begin{lemma}\label{lem:Gnq_close_eigs}
        Let $G(n,q)$ be a Lee graph with $q \geq 4$. Assume 
        $$
        (n,q) \notin \{(1,5),(2,5),(1,6),(2,7),(1,9)\}.
        $$
        Then, $G(n,q)$ has eigenvalues $\theta_1^*,\theta_2^*$ such that $0 \geq \theta_1^* > -1$ and $-1 \geq \theta_2^* \geq -2$.
    \end{lemma}
    \begin{proof}
        We split into four cases:
        \begin{enumerate}
            \item The case $q \notin \{5,6,7,9,10,11\}$. First, we show that $C_q$ has eigenvalues $\lambda_1^*,\lambda_2^*$ such that $1 > \lambda_1^* \geq 0 \geq \lambda_2^* >-1$ and $\lambda_1^*-\lambda_2^* < 1$. For $q \in [4,20] \setminus \{5,6,7,9,10,11\}$, we manually verify this holds. Now assume $q\geq 21$. Let $l_1 =\lfloor\frac{q}{4}\rfloor$ and observe:
            \begin{equation*}
                \frac{1}{2} > 2\cos(\frac{6\pi}{14}) \geq \zeta_q^{l_1}+\zeta_q^{q-l_1} \geq 0.
            \end{equation*}
            
            Similarly, we can define $l_2 = \lceil\frac{q}{4}\rceil$, and find
            \begin{equation*}
                0 \geq  \zeta_q^{l_2}+\zeta_q^{q-l_2}  \geq 2\cos(\frac{4\pi}{7}) > -\frac{1}{2}.
            \end{equation*}
    
            Thus, if we let $\lambda_i  = \zeta_q^{l_i}+\zeta_q^{q-l_i}$, we have that $\lambda_1^*,\lambda_2^*$ satisfy our assumptions, for $q \geq 21$.
    
            Now we will construct eigenvalues of $G(n,q)$ which satisfy $0 \geq \theta_1^* > -1$ and $-1 \geq \theta_2^* \geq -2$.
    
            First $\theta_1^*$. We consider eigenvalues of the form $\theta_a^* := a \lambda_1^* + (n-a) \lambda_2^*$ for $a \in [0,n]$. Clearly $\theta_n \geq 0 \geq \theta_0$. If $\theta_0 > -1$, we are done. Hence, assume that this is not the case and observe $\theta_a^* - \theta_{a-1} = \lambda_2^* - \lambda_1^* <1$. However, this means that for any open interval of length $1$ between $\theta_n$ and $\theta_0$, there must be some $a \in [1,n]$ such that $\theta_{a}$ lies within this interval. In particular this holds for $(-1,0)$ and thus we have found our desired eigenvalue $\theta_1^*$.
            
            The approach for $\theta_2^*$ is similar, but requires one more observation. Recall that any graph has an eigenvalue, say $\lambda_3$, for which $\lambda_3 \leq -1$, including $C_q$. Furthermore, we know that $\lambda_3 \geq -2$, since $C_q$ is $2$-regular. Now we consider eigenvalues of the form $\theta_a^* := a \lambda_1^* + (n-1-a) \lambda_2^* + \lambda_3$ for $a \in [0,n-1]$. Clearly $\theta_n \geq \lambda_3 \geq \theta_0$. We can assume $\theta_n \leq -2$ and $\theta_0 >-1$, since otherwise we are already done. Then, by the same reasoning as above, there must exist some $a \in [0,n-1]$ such that $\theta_{a}$ lies within $[-2,-1]$. This gives us $\theta_2^*$.
            
            \item The case $q \in \{5,6,7,9,10,11\}$. Note that $C_q$ has eigenvalues $\zeta_q^{l} + \zeta_q^{q-l}$ for $l\in [1,q]$. In particular we observe that the sum of all $q$ of these eigenvalues is equal to
            \begin{equation*}
                \sum_{l=1}^q \zeta_q^{l} + \zeta_q^{q-l} = 2 \sum_{l=1}^q \zeta_q^l = 0.
            \end{equation*}
            Hence, if  $\theta_1^*,\theta_2^*$ are some desired eigenvalues of $G(n,q)$ for some $n$, then by adding the above sum, we can see that they are also eigenvalues of $G(n+q,q)$. 

            We can manually verify that for $q \in \{10,11\}$ we have such $\theta_1^*,\theta_2^*$ for $n \in \{1,\hdots q\}$, for $q \in \{6,9\}$ we have them for $n \in \{2,\hdots,q+1\}$ and for $q \in \{5,7\}$ we have them for $n \in \{3,\hdots,q+2\}$. Hence, by induction we find that they always exist for all larger $n$. The only case still left is $q=7,n=1$, where we can also manually verify that the graph $G(1,7)$ has the desired $\theta_1^*,\theta_2^*$. \qedhere
            \end{enumerate}
    \end{proof}

    In the rest of this section we deal with the problem of characterizing when $-1$ is an eigenvalue of $G(n,q)$, in terms of $n$ and $q$. We will see in Section \ref{sec:chit_Gnq} (Theorem \ref{thm:Gnq_chi2_behaviour})  that this problem naturally arises when trying to compute $\chi_2(G(n,q))$, which in turn is related to the existence of perfect codes in the Lee metric; we will see the latter in Theorem \ref{thm:kim_Gnq_perfect}. 
    For this purpose, we use a result by Lam and Leung \cite{lam2000} on vanishing sums of roots of unity.
    Define
    \begin{align*}
        S(q)&:=\left\{(x_0,\ldots,x_{q-1})\in \mathbb N^{q}\,:\, \sum_{i=0}^{q-1}x_i\zeta_q^i=0\right\}, \\
        W(q)&:=\{x_0+\cdots+x_{q-1}\,:\,(x_0,\ldots,x_{q-1})\in S(q) \}.
    \end{align*}
    
    \begin{theorem}[{\cite[Main Theorem]{lam2000}}]\label{thm:LamLeung}
        Let $q=p_1^{a_1}\cdots p_r^{a_r}$ be the prime factorization of $q$, with $p_i\neq p_j$ for $i\neq j$. Then
        \begin{equation*}
            W(q)=p_1 \mathbb{N} + \cdots + p_r \mathbb{N}.
        \end{equation*}
    \end{theorem}
    
    We now define similarly
    \begin{align*}
        S'(q)&:=\left\{(y_0,\ldots,y_{\lfloor\frac{q}{2}\rfloor})\in \mathbb N^{\lfloor\frac{q}{2}\rfloor+1}\,:\, \sum_{i=0}^{\lfloor\frac{q}{2}\rfloor}y_i(\zeta_q^i+\zeta_q^{q-i})=-1  \right\}, \\
        W'(q)&:=\{y_0+y_1+\cdots+y_{\lfloor\frac{q}{2}\rfloor}\,:\,(y_0,\ldots,y_{\lfloor\frac{q}{2}\rfloor})\in S'(q) \}.
    \end{align*}
    With this notation, we have that $-1$ is an eigenvalue of $G(n,q)$ if and only if $n\in W'(q)$.
    \begin{lemma}\label{lem:Gnq_-1_eig}
        Let $q=p_1^{a_1}\cdots p_r^{a_r}$ be the prime factorization of $q$, with $p_1<p_2<\cdots<p_r$. Then:
        \[
            W'(q) = \begin{cases}
                \dfrac{1}{2}\left(\left(-1 + 4\mathbb{N} + \sum_{i=2}^rp_i\mathbb N\right)\cap 2\mathbb N\right), &\text{ if } q \text{ even,}\\[14pt]
            \dfrac{1}{2}\left(\left(-1+\sum_{i=1}^rp_i\mathbb N\right)\cap 2\mathbb N\right), &\text{ if } q \text{ odd.}
        \end{cases}
        \]
    \end{lemma}
    \begin{proof}
    Denote
        \[
        T(q) := \begin{cases}
          \dfrac{1}{2}\left(\left(-1 + 4\mathbb{N} + \sum_{i=2}^rp_i\mathbb N\right)\cap 2\mathbb N\right), &\text{ if } q \text{ even,}\\[14pt]
        \dfrac{1}{2}\left(\left(-1+\sum_{i=1}^rp_i\mathbb N\right)\cap 2\mathbb N\right), &\text{ if } q \text{ odd.}  
        \end{cases}
    \]
    
    We consider the even and odd case separately.
    \begin{enumerate}
        \item Assume $q$ even, that is, $p_1=2$. Then, $y_{\lfloor\frac{q}{2}\rfloor} = y_{\frac{q}{2}}$. Now observe that       
        \begin{equation*}
            (y_0,\ldots,y_{\frac{q}{2}})\in S'(q) \implies (1+2y_0,y_1,\ldots, y_{\frac{q}{2}-1},2y_{\frac{q}{2}},y_{\frac{q}{2}-1},\ldots, y_1)\in S(q).
        \end{equation*}
        Hence,
        \begin{equation*}
            W'(q)\subseteq  T(q),
        \end{equation*}

        On the other hand, let $n\in T(q)$. Thus, $n$ is of the form 
        \begin{equation*}
            n = \frac{1}{2}\left(-1+\sum_{i=1}^ry_ip_i\right),
        \end{equation*}
        for some $y_1,\ldots,y_r \in \mathbb N$ such that $\sum_{i=1}^r y_i$ is odd and $y_1$ is even. In particular, there exists a $j \in \{1, \ldots, r\}$ such that $y_j \ge 1$. We now exhibit a linear combination of the $\zeta_q^i+\zeta_q^{q-i}$'s which equals $-1$ and whose coefficients sum to $n$. This is given, for instance, by
        \begin{align*}
            &\frac{y_1}{2}\underbrace{\left(\zeta_{2}+\zeta_{2}\right)}_{-2} +  \sum_{\ell=2}^ry_\ell \underbrace{\sum_{i=1}^{\frac{p_{\ell}-1}{2}}\left(\zeta_{p_\ell}^i+\zeta_{p_\ell}^{p_\ell-i}\right)}_{-1} +  \frac{1}{2}\left(\left(\sum_{\ell=1}^ry_\ell \right)-1\right)\underbrace{(\zeta_{p_r}^0+\zeta_{p_r}^{p_r})}_{2} \\
            = &-\left(\sum_{\ell=1}^ry_\ell\right)  +  \left(\sum_{\ell=1}^ry_\ell \right)-1 = -1.
        \end{align*}
        
        \item Assume $q$ odd. Then $y_{\lfloor\frac{q}{2}\rfloor} = y_{\frac{q-1}{2}}$. Now observe that 
        \begin{equation*}
            (y_0,\ldots,y_{\frac{q-1}{2}})\in S'(q) \implies (1+2y_0,y_1,\ldots, y_{\frac{q-1}{2}},y_{\frac{q-1}{2}},\ldots, y_1)\in S(q).
        \end{equation*}

        Hence,
        \begin{equation*}
            W'(q)\subseteq \frac{1}{2}\left(\left(W(q)-1\right)\cap 2\mathbb N\right) = T(q),
        \end{equation*}
        where the final equality follows from Theorem \ref{thm:LamLeung}.
    
        On the other hand, let $n\in T(q)$. Thus, $n$ is of the form 
        \begin{equation*}
            n = \frac{1}{2}\left(-1+\sum_{i=1}^ry_ip_i\right),
        \end{equation*}
        for some $y_1,\ldots,y_r \in \mathbb N$ such that $\sum_i^r y_i$ is odd. In particular, there exists a $j \in \{1, \ldots, r\}$ such that $y_j \ge 1$. We now exhibit a linear combination of the $\zeta_q^i+\zeta_q^{q-i}$'s which is $-1$ and whose coefficients sum to $n$. This is given, for instance, by
        \begin{align*}
            &\sum_{\ell=1}^ry_\ell \underbrace{\sum_{i=1}^{\frac{p_{\ell}-1}{2}}\left(\zeta_{p_\ell}^i+\zeta_{p_\ell}^{p_\ell-i}\right)}_{-1} +  \frac{1}{2}\left(\left(\sum_{\ell=1}^ry_\ell \right)-1\right)\underbrace{(\zeta_{p_r}^0+\zeta_{p_r}^{p_r})}_{2} \\
            = &-\left(\sum_{\ell=1}^ry_\ell\right)  +  \left(\sum_{\ell=1}^ry_\ell \right)-1 = -1.
        \end{align*}
    \end{enumerate}
    Thus, $W'(q) = T(q)$.
\end{proof}

We can use Lemma \ref{lem:Gnq_-1_eig} to derive some additional results, distinguishing when $q$ is a prime power or not.

\begin{corollary}
    Let $p_1$ be a prime number. If $q=p_1^{a_1}$, then
    \begin{equation*}
        W'(q)= \begin{cases}
            \emptyset &\text{ if } p_1=2, \\[6pt]
            \left\{\dfrac{1}{2}(-1+(1+2x)p_1) \,:\, x \in \mathbb N\right\}. &\text{ if } p_1>2.
        \end{cases}
    \end{equation*}
    In particular, there exist arbitrarily large $n$'s such that $-1$ is not an eigenvalue of $G(n,q)$.
\end{corollary}

When $q$ is not a prime power, things are different.
 First, recall the following well-known fact.

\begin{lemma}\label{lem:additive}
        For any $a,b\in \mathbb N$ with $\gcd(a,b)=1$, we have
    $$m\in a\mathbb N+b\mathbb N$$
    for every $m> (a-1)(b-1)$.
\end{lemma}

\begin{corollary}\label{cor:asymptotic_-1}
    Let $q=p_1^{a_1}\cdot \cdots \cdot p_r^{a_r}$, with $r>1$ and $p_1<\ldots <p_r$. Then, $-1$ is an eigenvalue of $G(n,q)$ for every 
    \begin{equation*}
        n > \begin{cases}
            \dfrac{-1+(2p_1-1)(p_2-1))}{2}, &\text{ if } q \text{ even,} \\[12pt]
            \dfrac{-1+(p_1-1)(p_2-1))}{2}, &\text{ if } q \text{ odd.}
        \end{cases}  
    \end{equation*} 
\end{corollary}

\begin{proof}
  By Lemma \ref{lem:Gnq_-1_eig}, we have that $W'(q)=T(q)$. Moreover, one can check that 
  \begin{equation*}
      T(q)\supseteq \begin{cases}
        \dfrac{1}{2}((-1+2p_1\mathbb N + p_2\mathbb N)\cap 2\mathbb N), &\text{ if } q \text{ even,} \\[12pt]
        \dfrac{1}{2}((-1+p_1\mathbb N + p_2\mathbb N)\cap 2\mathbb N), &\text{ if } q \text{ odd.}
      \end{cases}
  \end{equation*}
  Using Lemma \ref{lem:additive}, we deduce that $(p_1\mathbb N + p_2\mathbb N)\supseteq \{x\,:\, x>(p_1-1)(p_2-1)\}$, from which we conclude.
\end{proof}

    \subsection{Bounding the distance chromatic number of the Lee graph}\label{sec:chit_Gnq}

While Sopena and Wu \cite{sopena_coloring_2010} and Pór and Wood \cite{por_colourings_2009} have investigated the distance-$2$ chromatic number for products of cycles, they did this in a more general setting than $G(n,q)$, allowing the cycles to be also different. Kim and Kim \cite{kim_2-distance_2011} were the first to investigate $\chi_2(G(n,q))$. Their work was motivated by past research on $\chi_2(Q_n)=\chi_2(G(n,2))$ and the fact that there is a relation between $\chi_2(G(n,q))$ and perfect Lee codes. Kim and Kim approached the problem of finding $\chi_2(G(n,q))$ by using notions and results from coding theory, in particular perfect Lee codes.
    One simple lower bound on $\chi_2(G(n,q))$ follows  using the graph regularity.
    \begin{lemma}[{\cite[Lemma 2.1]{kim_2-distance_2011}}]
        $\chi_2(G(n,q)) \geq 2n+1$.
        \end{lemma}
    
    Using this connection to Lee codes, Kim and Kim \cite{kim_2-distance_2011} focused their study on the case $n=3$. For several smaller cases they were able to provide exact parameter values, as the following result shows.
    \begin{theorem}[{\cite[Theorems 3.2, 3.4, 3.5, 3.7, 3.8 and Corollary 3.9]{kim_2-distance_2011}}]\label{thm:kim_G3q}
        Let $q \geq 3$, then 
        \begin{equation*}
        \chi_2(G(3,q)) = \begin{cases}
            9 &\text{ if } q=3,5,6, \\
            8 &\text{ if } q=4l \text{ and }l \text{ not a multiple of } 7,\\
            7 &\text{ if } q=7l.
        \end{cases}
        \end{equation*}
        In particular, $\chi_2(G(3,q)) = 7$ if and only if $q=7l$.
    \end{theorem}

    \begin{corollary}[{\cite[Corollary 3.9]{kim_2-distance_2011}}]
        Assume that  $q=3l$ or $q=5l$. Then, $\chi_2(G(3,q)) \leq 9$.
    \end{corollary}

    In what follows, we will use the results from Section \ref{sec:Gnq_spec_prelims} on the spectrum of $G(n,q)$ to obtain new eigenvalue bounds on $\chi_2(G(n,q))$.  
    We start by investigating how the bound from Corollary \ref{cor:ratio_chi_2_reg} behaves when both $\theta_i$ and $\theta_{i-1}$ are close to $-1$.

     \begin{corollary}\label{cor:eig_eps_bound}
        Let $G$ be a $k$-regular graph with unique eigenvalues $\theta_0 > \cdots > \theta_d$. Let $\theta_i$ be the largest eigenvalue such that $\theta_i \leq -1$. If $\theta_i \geq -2$ and $\theta_{i-1} \leq 0$, then 
        \begin{equation*}
            \chi_2(G) \geq
            \begin{cases}
                k+1, &\text{ if } \theta_i = -1,\\
                k+2, &\text{ else.}
            \end{cases} 
        \end{equation*}
    \end{corollary}
    \begin{proof}
        If $\theta_i = -1$ we obtain:
        \begin{equation*}
            \chi_2(G) \geq \frac{(\theta_0 - \theta_i)(\theta_0 - \theta_{i-1})}{\theta_0 +\theta_i\theta_{i-1}} = \frac{(k +1)(k - \theta_{i-1})}{k -\theta_{i-1}} = k+1.
        \end{equation*}
        Now, assume $\theta_i < -1$. We write $\theta_i =-1-\epsilon_1$ and $\theta_{i-1}=-1+\epsilon_2$ for $0<\epsilon_1 \leq 1$ and $0<\epsilon_2 \leq 1$.
         By Corollary \ref{cor:ratio_chi_2_reg}, we get
            \begin{align*}
                \chi_2(G) \geq & \frac{(\theta_0 - \theta_i)(\theta_0 - \theta_{i-1})}{\theta_0 +\theta_i\theta_{i-1}}=\frac{(k+1+\epsilon_1 -\epsilon_2)(k+1) -\epsilon_1\epsilon_2}{k+1+\epsilon_1-\epsilon_2 -\epsilon_1\epsilon_2}\\
                = &\frac{(k+1+\epsilon_1 -\epsilon_2-\epsilon_1\epsilon_2)(k+1)+k\epsilon_1\epsilon_2}{k+1+\epsilon_1 -\epsilon_2-\epsilon_1\epsilon_2}>k+1.
            \end{align*}
        On the other hand, since $\alpha:=k+1+\epsilon_1-\epsilon_2-(k+1)\epsilon_1\epsilon_2\geq 0$, we have
        \begin{align*}\frac{(\theta_0 - \theta_i)(\theta_0 - \theta_{i-1})}{\theta_0 +\theta_i\theta_{i-1}}&=\frac{(k+1+\epsilon_1 -\epsilon_2)(k+1) -\epsilon_1\epsilon_2}{k+1+\epsilon_1-\epsilon_2 -\epsilon_1\epsilon_2}\\
        &\le \frac{(k+1+\epsilon_1 -\epsilon_2)(k+1) -\epsilon_1\epsilon_2+\alpha}{k+1+\epsilon_1-\epsilon_2 -\epsilon_1\epsilon_2}\\
        &=k+2.
        \end{align*}
        Combining the above inequalities, we obtain
        \begin{equation*}
            k+2 \geq \frac{(\theta_0 - \theta_i)(\theta_0 - \theta_{i-1})}{\theta_0 +\theta_i\theta_{i-1}} > k+1,
        \end{equation*}
        and hence
        \begin{equation*}
            \left\lceil\frac{(\theta_0 - \theta_i)(\theta_0 - \theta_{i-1})}{\theta_0 +\theta_i\theta_{i-1}}\right\rceil = k+2.\qedhere
        \end{equation*}
    \end{proof}

    Using Corollary \ref{cor:eig_eps_bound}, we can illustrate how the bound from Corollary \ref{cor:ratio_chi_2_reg} performs on $G(n,q)$, for $q \geq 4$, by showing that it has eigenvalues within a certain range.

\begin{theorem}\label{thm:Gnq_chi2_behaviour}
        If $q\ge 4$, then the distance-$2$ chromatic number of $G(n,q)$ satisfies:
        \begin{equation*}
            \chi_2(G(n,q)) \geq \begin{cases}
                2n+1 &\text{ if } -1 \text{ is an eigenvalue of } G(n,q),\\
                2n+2 &\text{ else}, \\
                2n+3 &\text{ if } n=1,q=5.
            \end{cases}
        \end{equation*}
    \end{theorem}
    \begin{proof}
        For $(n,q) \notin \{(1,5),(2,5),(1,6),(2,7),(1,9)\}$, we know from Lemma \ref{lem:Gnq_close_eigs} that $G(n,q)$ has eigenvalues $\theta_1^*,\theta_2^*$ such that $0 \geq \theta_1^* > -1$ and $-1 \geq \theta_2^* \geq -2$. 
        Applying Corollary \ref{cor:eig_eps_bound} to these cases gives us the desired result.

        This leaves us with the cases of $G(1,5), G(2,5), G(1,6), G(2,7)$ and $G(1,9)$. For these we can manually check to find the bounds:
        \begin{align*}
            \alpha_2(G(1,5)) \geq 5, \\ 
            \alpha_2(G(2,5)) \geq 5, \\
            \alpha_2(G(1,6)) \geq 3, \\
            \alpha_2(G(2,7)) \geq 6, \\
            \alpha_2(G(1,9)) \geq 3.
        \end{align*}
        Since $G(2,7)$ does not have $-1$ as an eigenvalue, and $G(2,5)$ and $G(1,6)$ and $G(1,9)$ do have $-1$ as an eigenvalue, these four cases also behave as the rest. This means that $G(1,5)$ is the only exception.
    \end{proof}

    Additionally, we can show that the bound from Corollary \ref{cor:ratio_chi_2_reg} is tight in certain cases, by using part of Theorem \ref{thm:kim_G3q} from Kim and Kim \cite{kim_2-distance_2011}.
    
    \begin{proposition}\label{thm:Gn7l_tight}
        Let $l \in \mathbb{N}^+$, then 
        \begin{equation*}  
        \chi_2(G(3,7l)) = \frac{(\theta_0 - \theta_i)(\theta_0 - \theta_{i-1})}{\theta_0 +\theta_i\theta_{i-1}}.
        \end{equation*}
    \end{proposition}
    \begin{proof}
        Lemma \ref{lem:Gnq_-1_eig} tells us $3 = \frac{1}{2} (-1+7) \in W'(7l)$, thus $-1$ is an eigenvalue of $G(3,7l)$. Hence, by Corollary \ref{cor:eig_eps_bound} we have $\chi_2(G(3,q)) \geq 7$, which Theorem \ref{thm:kim_G3q} tells us is tight.
    \end{proof}

\subsubsection*{Computational results}

In Tables \ref{tab:G3q_chi2}, \ref{tab:G3q_chi3}, \ref{tab:G3q_chi4} and \ref{tab:G4q_chi2} in the Appendix,  the results obtained by applying the bounds of Section \ref{sec:chit_Gnq} are presented. These were computed using \textsc{SageMath}. Where applicable, we also added the best known lower bounds on $\chi_t(G)$, obtained from the best known upper bounds on $A_q^L(n,d)$ as seen in \cite{astola_bounds_2013} and \cite{polak_semidefinite_2019}. If for some graph, computing the theoretical value, or solving the LP \eqref{eq:ratio_chi_LP_walk} took more than 30 minutes, the corresponding entry in the table is replaced by ``time''. Unfortunately, since the number of vertices of $G(n,q)$ blows up quite rapidly for $n>3$, we do not have many data points in Table \ref{tab:G4q_chi2}. Even applying the theoretically obtained bound from Corollary \ref{cor:ratio_chi_2_reg} becomes computationally intensive due to the spectrum, which is made up of sums of cosines. Indeed, in theory LP \eqref{eq:ratio_chi_LP_walk} should be able to handle quite large graphs (see \cite{abiad_eigenvalue_2023}), however for this the spectrum needs to be manually supplied, as the built in Sage function is rather slow. For comparison, we add the previously best known bounds from \cite{astola_bounds_2013} and \cite{polak_semidefinite_2019}.
We should note that computing our eigenvalue bounds (using the closed formulas from Corollaries \ref{cor:ratio_chi_2_reg} and \ref{cor:ratio_chi_3_reg}, or for larger $t$, LP \eqref{eq:ratio_chi_LP_walk}) is, for small graphs like the ones we tested,
significantly faster than solving an SDP (like the one proposed in \cite{polak_semidefinite_2019}), and in many cases our bounds perform fairly well, as shown in Tables \ref{tab:G3q_chi2}, \ref{tab:G3q_chi4}, \ref{tab:G3q_chi4}, \ref{tab:G4q_chi2}.

\section{A characterization of perfect Lee codes with minimum distance $3$}\label{sec:existenceperfectLeecodes}

  Let $B_r(c)=\{b \in \mathcal A_q^n : d_L(b,c) \leq r \}$ denote the ball of radius $r$ centered in $c$. Then, we define the \emph{packing radius} of $C$ as the largest $r$ such that $B_r(c)\cap B_r(c') = \emptyset$ for every $c,c'\in C$ with $c\neq c'$. We define the \emph{covering radius} of $C$ as the smallest $r$ such that $\bigcup_{c \in C} B_r(c) = \mathcal A_q^n$. A Lee code $C \subset \mathcal A_q^n$ is called a \emph{perfect Lee code} if the packing and covering radius of $C$ coincide. In other words, the balls centered in the codewords of a perfect Lee code with radius equal to the covering radius give a perfect tiling of the whole ambient space $\mathcal A_q^n$. Clearly, for this to be possible, $d$ must be odd.

    The problem of finding perfect Lee codes has a long history, dating back to 1968 when Golomb and Welch \cite{golomb_perfect_1970} made the following conjecture:
    \begin{conjecture}[{\cite[Section 7]{golomb_perfect_1970}}]\label{con:golomb_welch}
        Let $d$ be odd. There exist no perfect $(n,M,d)_q$-Lee codes for $n \geq 3$, $q \geq d \geq 5$.
    \end{conjecture}

This conjecture is still widely open.  We refer the interested reader to the extensive survey given by  
Horak and Kim \cite{horak_50_2018}.
In this section, however, we will focus on the existence of perfect Lee codes with minimum Lee distance $d=3$, which is not mentioned in Conjecture \ref{con:golomb_welch}. This is due to the fact that perfect $(n,M,3)_q$-Lee codes exist in some cases. This was already shown by Golomb and Welch themselves in \cite{golomb_perfect_1970}. The state-of-the-art on perfect Lee codes of minimum distance $d=3$ is discussed in \cite[Section B]{horak_50_2018}. In this regard, Kim and Kim \cite{kim_2-distance_2011} made several key observations on the existence of perfect $(n,M,3)_q$-Lee codes, and their connection to $\chi_2(G(n,q))$. Around the same time, AlBdaiwi, Horak and Milazzo \cite{AHM09} characterized the parameters $n$ and $q$ such that there exist a \emph{linear} perfect $(n,M,3)_q$-Lee code. In this section we completely settle the problem of the existence of perfect $(n,M,3)_q$-Lee codes, dropping the linearity assumption.

We first recall the following result of Kim and Kim, who gave a characterization of the parameters for which perfect Lee codes exists in terms of $\chi_2(G(n,q))$. 

\begin{theorem}[{\cite[Theorem 2.2]{kim_2-distance_2011}}]\label{thm:kim_Gnq_perfect}
        There exists a perfect $(n,M,3)_q$-Lee code if and only if $\chi_2(G(n,q)) = 2n+1$.
    \end{theorem}

This link allows us to derive nonexistence results of perfect $(n,m,3)_q$-Lee codes in terms of $n$ and $q$, using the eigenvalue results of Theorem \ref{thm:Gnq_chi2_behaviour}.

\begin{corollary}\label{cor:-1_nonexistence}
        Let $q\geq 4$ and let $W'(q)$ as defined in Lemma \ref{lem:Gnq_-1_eig}. If $n \notin W'(q)$ then there is no perfect $(n,M,d)_q$-Lee code.
    \end{corollary}

\begin{proof}
Since $n \notin W'(q)$, $G(n,q)$ does not have $-1$ as an eigenvalue.  
    Theorem \ref{thm:Gnq_chi2_behaviour} tells us that if $q \geq 4$ and $-1$ is not an eigenvalue of $G(n,q)$, then $\chi_2(G(n,q)) \geq 2n+2$. We conclude the proof using Theorem \ref{thm:kim_Gnq_perfect}.
\end{proof}

Unfortunately, the condition given in Corollary \ref{cor:-1_nonexistence} for the nonexistence of perfect Lee codes of minimum distance $d=3$ is not improving on the previously known results. In particular, it is weaker than the following result.

\begin{proposition}[{\cite[Corollary 2.5]{kim_2-distance_2011}}]\label{prop:kim_Gnq_nonexistence}
        If there exists a perfect $(n,M,3)_q$-Lee code, then $2n+1$ divides $q^n$.  
\end{proposition}

Nevertheless, next we show that we can completely characterize the values of $n$ and $q$ such that there exist perfect $(n,M,3)_q$-Lee codes. We will do it by exploiting the following characterization results concerning \emph{linear} Lee codes obtained in \cite{AHM09}. We can identify the alphabet $\mathcal A_q$ with the ring $\mathbb Z/q\mathbb Z$ of integers modulo $q$. A \emph{linear} code is a subset $C\subseteq (\mathbb Z/q\mathbb Z)^n$ which is also a $\mathbb Z/q\mathbb Z$-submodule.

\begin{theorem}[{\cite[Theorem 15]{AHM09}}]\label{thm:perfect_linear_codes}
        Let $2n+1=r_1^{a_1} \cdots  r_t^{a_t}$ be the prime factorization of $2n+1$.
        Then, there exists a linear perfect $(n,M,3)_q$-Lee code if and only if  $r_1\cdots r_t$ divides $q$.
      \end{theorem}

We can now put everything together, and show the main result of this section: a characterization of when a perfect Lee code of minimum distance $d=3$ exists. In terms of the distance-$2$ chromatic number of $G(n,q)$, this also improves on Theorem \ref{thm:Gnq_chi2_behaviour}.

\begin{theorem}\label{thm:golomb_d=3}
    Let $n,q \in \mathbb N$. The following are equivalent.
    \begin{enumerate}[label=\normalfont{(\arabic*)}]
        \item There exists a perfect $(n,M,3)_q$-Lee code.
        \item $2n+1$ divides $q^n$.
        \item $r_1\cdots r_t$ divides $q$, where $2n+1=r_1^{a_1}\cdots r_t^{a_t}$ is the prime factorization of $2n+1$. 
        \item $\chi_2(G(n,q))=2n+1$.
    \end{enumerate}
\end{theorem}

     \begin{proof}
      By Theorem \ref{thm:perfect_linear_codes} and Proposition \ref{prop:kim_Gnq_nonexistence}, it is enough to show the equivalence between (2) and (3). For this purpose, write $q=s_1^{b_1}\cdots s_{\ell}^{b_\ell}$. 
        
         If $r_1\cdots r_t$ divides $q$, then this means that $t\le \ell$ and up to reordering the $s_i$'s, we may assume $r_i=s_i$ for each $i=1,\ldots, t$. The claim follows observing that $a_i\le n$ for every $i$.

 On the other hand, if $2n+1=r_1^{a_1}\cdots r_t^{a_t}$ divides $q^n$, then, $r_1\cdots r_t$ must divide $s_1\cdots s_\ell$, and thus, also $q$.
     \end{proof}

Observe that by plugging $n=3$ in Theorem \ref{thm:golomb_d=3}, we immediately derive the second part of the statement in Theorem \ref{thm:kim_G3q}.

\bigskip

\subsection*{Acknowledgements}
Aida Abiad is supported by the Dutch Research Council through the grant VI.Vidi.213.085. Alessandro Neri is supported by the Research Foundation -- Flanders (FWO) through the grant 12ZZB23N.
The authors thank Sjanne Zeijlemaker for the discussions on the LP implementation for Theorem \ref{thm:ratio_chi}.


\newpage

\section*{Appendix}\label{appendix}

    \begin{table}[ht]
        \centering
        \footnotesize
        \subfloat[Performance of bounds for $\chi_2(Q_n)$.]{\begin{tabular}{ccc}
        \hline
        $n$ & Corollary \ref{cor:Qn_chi2_ngobound} \cite{ngo_new_2002}& $\chi_2(Q_n)$ \\
        \hline
         $\mathbf{2} $ & $\mathbf{4}$ & $4$ \cite{wan_near-optimal_1997} + \cite{best_triply_1977}\\
         $\mathbf{3}$ & $\mathbf{4}$ & $4$ \cite{wan_near-optimal_1997} + \cite{best_triply_1977}\\
         $\mathbf{4}$ & $\mathbf{8}$ & $8$ \cite{wan_near-optimal_1997} + \cite{best_triply_1977}\\
         $5$ & $7$ & $8$ \cite{wan_near-optimal_1997} + \cite{best_triply_1977}\\
         $\mathbf{6}$ & $\mathbf{8}$ & $8$ \cite{wan_near-optimal_1997} + \cite{best_triply_1977}\\
         $\mathbf{7}$ & $\mathbf{8}$ & $8$ \cite{wan_near-optimal_1997} + \cite{best_triply_1977}\\
         $8$ & $11$ & $13$ \cite{kokkala_chromatic_2017} \\
         $9$ & $11$ & $\in\{13,14\}$ \\
         $10$ & $13$ & $\in\{15,16\}$ \\
         $11$ & $13$ & $\in\{15,16\}$ \\
         $12$ & $15$ & $16$ \cite{wan_near-optimal_1997} + \cite{best_triply_1977}\\
         $13$ & $15$ & $16$ \cite{wan_near-optimal_1997} + \cite{best_triply_1977}\\
         $\mathbf{14}$ & $\mathbf{16}$ & $16$ \cite{wan_near-optimal_1997} + \cite{best_triply_1977}\\
         $\mathbf{15}$ & $\mathbf{16}$ & $16$ \cite{wan_near-optimal_1997} + \cite{best_triply_1977}\\
         \hline
        \end{tabular}
        \label{tab:Qn_t2}} \hspace{10mm}
        \centering
        \subfloat[Performance of bounds for $\chi_3(Q_n)$.]{        
         \hspace{10mm}
         \begin{tabular}{cccc}
        \hline
        $n$ & Corollary \ref{cor:Qn_chi3_ngobound} \cite{ngo_new_2002} & $\chi_3(Q_n)$ \\
        \hline
         $\mathbf{3}$ & $\mathbf{8}$ & $8$ \\
         $\mathbf{4}$ & $\mathbf{8}$ & $8$\\
         $5$ & $16$ & time\\
         $6$ & $13$ & time \\
         $7$ & $16$ & time \\
         $8$ & $16$ & time \\
         $9$ & $21$ & time \\
         $10$ & $21$ & time \\
         $11$ & $25$ & time \\
         $12$ & $25$ & time\\
         $13$ & $29$ & time\\
         $14$ & $29$ & time\\
         $15$ & $32$ & time\\
         \hline
         \ & \ & \ \\
        \end{tabular}
        \hspace{10mm}
        \label{tab:Qn_t3}} \\
        \centering
        \subfloat[Performance of bounds for $\chi_4(Q_n)$.]{\begin{tabular}{ccccc}
        \hline
        $n$ & Corollary \ref{cor:Qn_chit_ratiobound} & Theorem \ref{thm:Qn_chit_ngobound} \cite{ngo_new_2002} & $\chi_4(Q_n)$ \\
        \hline
         $\mathbf{4}$ & $\mathbf{16}$ & $\mathbf{16}$ & $16$ \\
         $\mathbf{5}$ & $\mathbf{16}$ & $\mathbf{16}$ & $16$ \\
         $6$ & $32$ & $32$ & time \\
         $7$ & $43$ & $43$ & time \\
         $8$ & $43$ & $43$ & time \\
         $9$ & $57$ & $52$ & time \\
         $10$ & $57$ & $69$ & time \\
         $11$ & $79$ & $69$ & time \\
         $12$ & $90$ & $86$ & time \\
         $13$ & $102$ & $106$ & time \\
         $14$ & $121$ & $107$ & time \\
         $15$ & $127$ & $128$ & time \\
         \hline
        \end{tabular}
        \label{tab:Qn_t4}} \hspace{5mm}
        \centering
        \subfloat[Performance of bounds for $\chi_5(Q_n)$.]{\begin{tabular}{ccccc}
        \hline
        $n$ & Corollary \ref{cor:Qn_chit_ratiobound} & Theorem \ref{thm:Qn_chit_ngobound} \cite{ngo_new_2002} & $\chi_5(Q_n)$ \\
        \hline
         $\mathbf{5}$ & $\textbf{32}$ & $\textbf{32}$ &  $32$ \\
         $6$ & $32$ & $32$ & $32$ \\
         $7$ & $64$ & $64$ & time \\
         $8$ & $86$ & $86$ & time \\
         $9$ & $86$ & $86$ & time \\
         $10$ & $114$ & $103$ & time \\
         $11$ & $114$ & $137$ & time \\
         $12$ & $158$ & $137$ & time \\
         $13$ & $179$ & $171$ & time \\
         $14$ & $203$ & $211$ & time \\
         $15$ & $241$ & $213$ & time \\
         \hline
         \ & \ & \ & \\
        \end{tabular}
        \label{tab:Qn_t5}}
        \caption{Performance of bounds for $\chi_t(Q_n)$.}
    \end{table}


 \begin{table}[htp!]
        \centering
        \footnotesize
        \subfloat[Performance of the Ratio-type bound for $\chi_2(G)$ on $G(3,q)$.]{\begin{tabular}{cccccc}
            \hline
            Graph & Corollary \ref{cor:ratio_chi_2_reg}  & $\chi_2$ \\
            \hline
             $\mathbf{G(3,3)}$ & $\mathbf{9}$ &  $9$ \\
             $\mathbf{G(3,4)}$ & $\mathbf{8}$& $8$ \\
             $G(3,5)$ & $8$ &$9$ \\
             $G(3,6)$ & $8$& $9$ \\
             $\mathbf{G(3,7)}$ & $\mathbf{7}$ & $7$ \\
             $\mathbf{G(3,8)}$ & $\mathbf{8}$ & $8$ \\
             $G(3,9)$ & $8$ & $\leq 9$ \\
             \hline
        \end{tabular}
        \label{tab:G3q_chi2}}\hspace{15mm}
        \centering
        \subfloat[Performance of the Ratio-type bound for $\chi_3(G)$ on $G(3,q)$.]{\begin{tabular}{ccccccc}
            \hline
            Graph & Corollary \ref{cor:ratio_chi_3_reg} & Best LB & $\chi_3$ \\
            \hline
            $\mathbf{G(3,3)}$ &  $\mathbf{27}$ & N/A & $27$ \\
            $G(3,4)$ & $13$   & N/A  & time \\
            $G(3,5)$ & $16$ & $18$ \cite{astola_bounds_2013} & time\\
            $G(3,6)$ & $12$   & $16$ \cite{polak_semidefinite_2019} & time\\
            $G(3,7)$ & $14$  & $17$ \cite{polak_semidefinite_2019} & time\\
            $G(3,8)$ & $13$  & N/A  & time \\
            $G(3,9)$ & $13$  & N/A  & time \\
             \hline
        \end{tabular}
        \label{tab:G3q_chi3}}\\
        \centering
        \subfloat[Performance of the Ratio-type bound for $\chi_4(G)$ on $G(3,q)$.]{\begin{tabular}{cccccc}
            \hline
            Graph & LP \eqref{eq:ratio_chi_LP_walk} &  Best LB & $\chi_4$ \\
            \hline
             $\mathbf{G(3,3)}$ & $1$ &  N/A & $27$ \\
             $G(3,4)$ & $32$ & N/A & time \\
             $G(3,5)$ & $32$& $42$ \cite{astola_bounds_2013} & time  \\
             $G(3,6)$ & $27$ & $36$ \cite{astola_bounds_2013} & time  \\
             $G(3,7)$ & $27$  & $35$ \cite{polak_semidefinite_2019} & time  \\
             $G(3,8)$ & $25$  & N/A & time \\
             $G(3,9)$ & time  & N/A & time \\
             \hline
        \end{tabular}
        \label{tab:G3q_chi4}}\hspace{5mm}
        \centering
        \subfloat[Performance of the Ratio-type bound for $\chi_2(G)$ on $G(4,q)$.]{\begin{tabular}{cccccc}
            \hline
            Graph & Corollary \ref{cor:ratio_chi_2_reg} & Best LB & $\chi_2$ \\
            \hline
             $\mathbf{G(4,3)}$ & $\mathbf{9}$ & N/A & $9$ \\
             $G(4,4)$ & $11$  & N/A & time \\
             $G(4,5)$ & $10$   & $11$ \cite{polak_semidefinite_2019}& time \\
             $G(4,6)$ & $9$   & $9$ \cite{astola_bounds_2013} & time \\
             \hline \\
             \ \\
             \ \\
             \end{tabular}
        \label{tab:G4q_chi2}}
        \caption{Performance of the Ratio-type bound for $\chi_t(G)$ on $G(n,q)$ compared to the best known lower bounds from \cite{astola_bounds_2013} and \cite{polak_semidefinite_2019}.}
    \end{table}

\end{document}